\documentclass{amsart}

\usepackage{amsfonts,amssymb,amsmath,amsthm}
\usepackage{tikz}
\usepackage{tikz-cd}

\newcommand{\kk}{\mathbb{K}}

\newcommand{\m}{\mathbf{m}}

\newcommand{\Q}{\mathcal{Q}}
\newcommand{\cD}{\mathcal{D}}

\newcommand{\cJ}{\mathcal{J}}
\newcommand{\cR}{\mathcal{R}}
\newcommand{\R}{\mathbb{R}}

\newcommand{\Z}{\mathbb{Z}}
\newcommand{\A}{\mathcal{A}}
\newcommand{\B}{\mathcal{B}}
\newcommand{\cS}{\mathcal{S}}
\DeclareMathOperator{\rk}{rk}
\DeclareMathOperator{\syz}{syz}
\newcommand{\reg}{\mbox{\emph{reg} }}
\newcommand{\lt}{L_2^{\text{\tiny{trip}}}}
\newcommand{\V}{\mathcal{V}}
\newcommand{\dx}{\frac{\partial}{\partial x}}
\newcommand{\dy}{\frac{\partial}{\partial y}}

\newtheorem{thm}{Theorem}[section]
\newtheorem{cor}[thm]{Corollary}
\newtheorem{lem}[thm]{Lemma}
\newtheorem{prop}[thm]{Proposition}

\theoremstyle{definition}
\newtheorem{defn}[thm]{Definition}
\newtheorem{exm}[thm]{Example}
\newtheorem{remark}[thm]{Remark}

\newtheorem{ques}[thm]{Question}

\usepackage{hyperref}

\begin{document}

\title{A homological characterization for freeness of multi-arrangements}
\author{Michael DiPasquale}
\address{Michael DiPasquale\\     
	Department of Mathematics\\     
	Oklahoma State University\\     
	Stillwater\\
	OK \ 74078-1058\\     
	USA}     
\email{mdipasq@okstate.edu}
\urladdr{\url{http://math.okstate.edu/people/mdipasq/}}

\maketitle

\begin{abstract}
Building on work of Brandt and Terao in their study of $k$-formality, we introduce a co-chain complex associated to a multi-arrangement and prove that its cohomologies determine freeness of the associated module of multi-derivations.  This provides a new homological method for determining freeness of arrangements and multi-arrangements.  We work out many applications of this homological method.  For instance, we prove that if a multi-arrangement is free then the underlying arrangement is $k$-formal for all $k\ge 2$.  We also use this method to completely characterize freeness of certain families of multi-arrangements in moduli, showcasing how the geometry of multi-arrangements with the same intersection lattice may have considerable impact on freeness.  New counter-examples to Orlik's conjecture also arise in connection to this latter analysis.


\end{abstract}

\setcounter{tocdepth}{1}
\tableofcontents

\section{Introduction}
A central hyperplane arrangement, which we will denote by $\A$, is a union of hyperplanes passing through the origin in a vector space $V\cong\kk^\ell$, where $\kk$ is a field.  Write $S$ for the symmetric algebra of $V^*$, which is isomorphic to a polynomial ring in $\ell$ variables.  Then $\A$ is the union of the zero-locus of linear forms $\alpha_H$, one for each hyperplane $H$ in $\A$.  The module of logarithmic $\A$-derivations, denoted $D(\A)$, consists of derivations $\theta\in \mbox{Der}_\kk(S)$ satisfying $\theta(\alpha_H)\in \alpha_H S$ for every $H\in\A$.  Study of this module was initiated by Saito~\cite{S80}; it is of particular interest to know when $D(\A)$ is a free $S$-module.  In this case $\A$ is called a free arrangement.  One of the central open questions in the theory of hyperplanes, due to Terao, is whether freeness of an arrangement is combinatorial, meaning that it can be detected from the lattice of intersections.

Let $\m:\A\rightarrow \Z_{>0}$ be a function, called a multiplicity, associating to each hyperplane $H$ a positive integer $\m(H)$; the pair $(\A,\m)$ is called a multi-arrangement.  The module of derivations of $(\A,\m)$, denoted $D(\A,\m)$, consists of those derivations $\theta\in \mbox{Der}_\kk(S)$ satisfying $\theta(\alpha_H)\in \alpha_H^{\m(H)}S$ for every $H\in\A$.  If $D(\A,\m)$ is a free $S$-module we say $(\A,\m)$ free and $\m$ is a free multiplicity of $\A$.  Due to a criterion stated by Ziegler~\cite{ZieglerMulti} and later improved by Yoshinaga~\cite{YoshCharacterizationFreeArr}, freeness of multi-arrangements is closely linked to freeness of arrangements.

There have been major advances in the understanding of multi-arrangements during the last decade.  The characteristic polynomial has been defined for multi-arrangements by Abe, Terao, and Wakefield~\cite{TeraoCharPoly} and they show that Terao's factorization theorem holds for this characteristic polynomial.  
Moreover, the addition-deletion theorem has also been extended by Abe, Terao, and Wakefield to multi-arrangements~\cite{EulerMult}.  This improved theory of multi-arrangements has recently led to remarkable progress in understanding freeness of arrangements and of Terao's question in particular~\cite{AbeDivisional,AbeDeletion}.

In this paper we add to the list of available tools for studying multi-arrangements by introducing a homological characterization for freeness.  The characterization involves building a co-chain complex which we denote $\cD^\bullet(\A,\m)$ from modules constructed by Brandt and Terao~\cite{BrandtTerao} to study $k$-formality (see Definition~\ref{defn:DerivationComplex} for details).  Chain complexes having very similar properties to $\cD^\bullet(\A,\m)$ appear in the theory of algebraic splines~\cite{Homology,LCoho}; applying techniques of Schenck and Stiller~\cite{Spect,CohVan} yields our main result, stated below.

\begin{thm}[Homological characterization of freeness]\label{thm:Free}
The multi-arrangement $(\A,\m)$ is free if and only if $H^k(\cD^\bullet(\A,\m))=0$ for $k> 0$.  Moreover, $D(\A,\m)$ is locally free if and only if $H^k(\cD^\bullet(\A,\m))$ has finite length for all $k>0$.
\end{thm}

Weaker versions of this statement have been proved recently and used to classify free multiplicities on several rank three arrangements~\cite{GSplinesGraphicArr,A3MultiBraid,X3}.  For simple arrangements, the forward direction of the first statement in Theorem~\ref{thm:Free} follows from work of Brandt and Terao~\cite{BrandtTerao}.  Homological methods are not new in the study of freeness of arrangements; besides the aforementioned work of Brandt and Terao, Yuzvinsky developed and studied the theory of cohomology of sheaves of differentials on arrangement lattices to great effect in~\cite{YuzCohoLocal,YuzFormal,YuzModuli}.  While we will not attempt to generalize this framework to multi-arrangements, Yuzvinsky's work, along with Brandt and Terao's, is an important motivation for this paper.

The remainder of the paper is devoted to applications of this homological criterion.  In \S~\ref{sec:ChainComplex} we extend a combinatorial bound on projective dimension of $D(\A,\m)$ due to Kung and Schenck in the case of simple arrangements.  In \S~\ref{sec:Formal} we elucidate the connection to $k$-formality and use the homological characterization of Theorem~\ref{thm:Free} to extend a result of Brandt and Terao~\cite{BrandtTerao} to multi-arrangements in Corollary~\ref{cor:MultifreeFormal}.



Following the initial applications of this homological characterization of freeness, we describe in \S~\ref{sec:Computations} how the chain complex $\cD^\bullet(\A,\m)$ can be concretely computed.  We have implemented this construction in the computer algebra system Macaulay2~\cite{M2}.  The code for constructing the chain complex, as well as a file working through many of the examples in this paper, may be found on the author's website: \href{https://math.okstate.edu/people/mdipasq/Research/Research.html}{math.okstate.edu/$\sim$mdipasq}.  In \S~\ref{sec:Computations} we also explicitly work out the structure of $\cD^\bullet(\A,\m)$ for graphic arrangements and show that Theorem~\ref{thm:Free} recovers the main result of~\cite{GSplinesGraphicArr}.

In \S~\ref{sec:TF2}, we study a class of arrangements which we call $TF_2$ arrangements; these are formal arrangements whose relations of length three are linearly independent.  We believe this study is well-motivated by the interesting behavior of multi-$TF_2$ arrangements in moduli as well as additional counter-examples to Orlik's conjecture which arise in the process.  We illustrate this in \S~\ref{sec:Examples} before proceeding to the body of the paper.  If $\A$ is a $TF_2$ arrangement, freeness of $(\A,\m)$ is determined by the vanishing of the single cohomology module $H^1(\cD^\bullet(\A,\m))$, making these arrangements well-suited to the homological methods afforded by Theorem~\ref{thm:Free}.  We show that a $TF_2$ arrangement is free if and only if it is supersolvable.  We completely classify free multiplicities on non-free $TF_2$ arrangements in Proposition~\ref{prop:H2Pres} and Theorem~\ref{thm:FreeMultNonFreeTF2}.  Moreover, we show that free multiplicities of free $TF_2$ arrangements can be determined in a combinatorial fashion from the exponents of its rank two sub-arrangements in Theorem~\ref{thm:FreeMultFreeTF2}.

We also give in \S~\ref{sec:SyzygiesTeraoConj} a syzygetic criterion for freeness of a multi-arrangement of lines, generalizing a criterion for freeness of $A_3$ multi-arrangements from~\cite{A3MultiBraid}.  Specializing to simple line arrangements gives an equivalent formulation of Terao's question for line arrangements, phrased in terms of syzygies of a certain module presented by a matrix of linear forms (Question~\ref{ques}).

\textbf{Acknowledgements:} I am indebted to Stefan Tohaneanu for pointing out his paper~\cite{StefanFormal}, which provided the inspiration to generalize the homological arguments in~\cite{GSplinesGraphicArr}.  The current work would not be possible without the collaboration of Chris Francisco, Jeff Mermin, Jay Schweig, and Max Wakefield on previous papers~\cite{A3MultiBraid,X3}.  Takuro Abe has been a consistent source of inspiring discussions and many patient explanations via e-mail.  Computations in the computer algebra system Macaulay2~\cite{M2} were very useful at all stages of research.

\subsection{Examples}\label{sec:Examples}
In this section we illustrate results which can be obtained by applying the homological criterion for freeness (Theorem~\ref{thm:Free}).  The three examples in this section are $TF_2$ arrangements, the definition and analysis of which appears in \S~\ref{sec:TF2}.

\begin{exm}\label{ex:multiplicitiessupersolvable}
Consider the line arrangement $\A(\alpha,\beta)$ defined by $xyz(x-\alpha z)(x-\beta z)(y-z)$ where $\alpha,\beta\in\kk$.  See Figure~\ref{fig:multiplicitiessupersolvable} for a projective picture of this arrangement over $\R$.  Clearly if $\alpha\neq\beta,\alpha\neq0,$ and $\beta\neq 0$, then the intersection lattice $L(\A(\alpha,\beta))$ does not change.  In fact, the arrangements $\A(\alpha,\beta)$ with $\alpha\neq\beta,\alpha\neq0,$ and $\beta\neq 0$ comprise the moduli space of this lattice (see Appendix~\ref{app:Moduli} for a brief summary of the moduli space of a lattice).  It is easily checked that $\A(\alpha,\beta)$ is supersolvable.

We will see in Theorem~\ref{thm:FreeMultFreeTF2} that the freeness of the multi-arrangement $(\A(\alpha,\beta),\m)$ can be determined if the exponents of the rank two sub multi-arrangements are known.  Write $\m(x),\m(y),\ldots$ for the multiplicity assigned to, respectively, $x=0,y=0,\ldots$.  There are two rank-two sub multi-arrangements of $(\A(\alpha,\beta),\m)$ defined by
\[
\begin{array}{rl}
\tilde{X}_1= & y^{\m(y)}z^{\m(z)}(y-z)^{\m(y-z)}\mbox{ and} \\
\tilde{X}_2= & x^{\m(x)}z^{\m(z)}(x-\alpha z)^{\m(x-\alpha z)}(x-\beta z)^{\m(x-\beta z)}.\\ 
\end{array}
\]
In Example~\ref{ex:multiplicitiessupersolvable0}, we deduce from Theorem~\ref{thm:FreeMultFreeTF2} that $(\A(\alpha,\beta),\m)$ is free if and only if either $\tilde{X}_1$ or $\tilde{X}_2$ has $\m(z)$ as an exponent.  This property is sensitive to the characteristic of $\kk$; we will assume in the remainder of this example that $\kk$ has characteristic zero.

Write $M_1=\m(y)+\m(z)+\m(y-z)$ and $M_2=\m(x)+\m(z)+\m(x-\alpha z)+\m(x-\beta z)$.  If $\kk$ has characteristic zero, the exponents of the multi-arrangement $\tilde{X}_1$ are known~\cite{Wakamiko}; $\m(z)$ is an exponent if and only if $M_1\le 2\m(z)+1$.  So we assume $M_1>2\m(z)+1$ and determine when $\tilde{X}_2$ has an exponent of $\m(z)$.

It is not difficult to show that if $\m(z)$ is an exponent of $\tilde{X}_2$, then $\m(z)=\max\{\m(x),\m(z),\m(x-\alpha z),\m(x-\beta z)\}$ (see Lemma~\ref{lem:Boolean}).  From~\cite{DerProjLine} it is known that $\m(z)$ is an exponent of $\tilde{X}_2$ if $M_2\le 2\m(z)+1$.  Moreover it follows from~\cite[Theorem~1.6]{AbeFreeness3Arrangements} that $\m(z)$ is not an exponent of $\tilde{X}_2$ if $M_2>2+2\m(z)$ (this also requires that $\kk$ has characteristic zero).  However if $M_2=2+2\m(z)$ then it is only known that $\m(z)$ is not an exponent of $\tilde{X}_2$ for \textit{generic} choices of $\alpha$ and $\beta$ (at least if $\kk=\mathbb{C}$~\cite{DerProjLine}).

To see what can happen if $M_2=2+2\m(z)$, consider the multi-arrangement $(\A(\alpha,\beta),\m)$ defined by
\[
x^3y^3z^3(x-\alpha z)(x-\beta z)(y-z)^3.
\]
Then $\tilde{X}_1=x^3y^3(y-z)^3$ and $\tilde{X}_2=x^3z^3(x-\alpha z)(x-\beta z)$.  The exponents of $\tilde{X}_1$ are $(4,5)$, while the exponents of $\tilde{X}_2$ are $(4,4)$ if $\alpha\neq-\beta$ and $(3,5)$ if $\alpha=-\beta$ (see~\cite{ZieglerMulti} or Lemma~\ref{lem:nn11exp}).  By Theorem~\ref{thm:FreeMultFreeTF2}, $(\A(\alpha,\beta),\m)$ is free if and only if $\alpha=-\beta$.

As a consequence, we see that for a fixed multiplicity $\m$ the free multi-arrangements $(\A,\m)$ in the moduli space of $L(\A)$ can form a non-empty proper Zariski closed subset, even when $\A$ is supersolvable over a field of characteristic zero.  In contrast, Yuzvinsky has shown that free arrangements form a Zariski open subset of the moduli space of $L(\A)$~\cite{YuzModuli}.
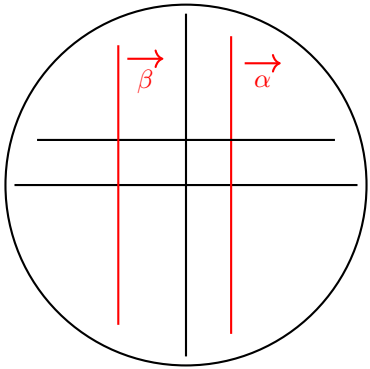
\begin{figure}
	
	\begin{tikzpicture}[scale=.6]
	
	\draw[thick] (0,0) circle (4cm);
	\draw[thick] (-3.8,0)--(3.8,0);
	\draw[thick] (0,-3.8)--(0,3.8);
	\draw[red,thick] (1,3.3)--(1,-3.3);
	\draw[thick] (3.3,1)--(-3.3,1);
	\draw[red,thick] (-1.5,-3.1)--(-1.5,3.1);
	\draw[->,red,thick] (1.3,2.7)--(2.1,2.7);
	\draw[->,red,thick] (-1.3,2.8)--(-.5,2.8);
	\node[red] at (1.7,2.3) {$\alpha$};
	\node[red] at (-.9,2.3) {$\beta$};
	
	\end{tikzpicture}
	\caption{A projective picture emphasizing the moduli in Example~\ref{ex:multiplicitiessupersolvable}}\label{fig:multiplicitiessupersolvable}
\end{figure}
\end{exm}

\begin{exm}\label{ex:TotallyNonFree}
Let $\A(\alpha,\beta)$ be the arrangement with defining polynomial $\Q(\A(\alpha,\beta))=xyz(x-\alpha y)(x-\beta y)(y-z)(x-z)$, where $\alpha,\beta\in\kk$.  See Figure~\ref{fig:TotallyNonFree} for a projective drawing of this arrangement over $\mathbb{R}$.  It is straightforward to show that if $\alpha\neq 1,\beta\neq 1,$ and $\alpha\neq\beta$, then the lattice $L(\A(\alpha,\beta))$ does not change.  Just as in Example~\ref{ex:multiplicitiessupersolvable}, these arrangements comprise the moduli space of this lattice.  It is easily checked that $\A(\alpha,\beta)$ is not free for any choice of $\alpha,\beta$ since its characteristic polynomial does not factor.

We will see in Theorem~\ref{thm:FreeMultNonFreeTF2} that if $\kk$ has characteristic $0$, the multi-arrangement $(\A(\alpha,\beta),\m)$ is free if and only if its defining equation has the form
\[
\Q(\A,\m)=x^ny^nz^n(x-\alpha y)(x-\beta y)(y-z)(x-z),
\]
where $n>1$ is an integer and $\alpha^{n-1}=\beta^{n-1}\neq 1$.  In particular, if $\alpha/\beta$ is not a root of unity in $\kk$, then $\A$ is totally non-free, meaning it does not admit any free multiplicities.  For instance, if $\kk=\R$, then $\A$ admits a free multiplicity if and only if $\alpha=-\beta$ (the free multiplicities occur precisely when $n>1$ is odd).  Since the arrangements $\A(\alpha,\beta)$ with $\alpha\neq 1,\beta\neq 1,$ and $\alpha\neq\beta$ all have the same intersection lattice, this shows that the property of being totally non-free is not combinatorial.  In contrast, Abe, Terao, and Yoshinaga have shown that the property of being totally free is combinatorial~\cite{TeraoTotallyFree}.

\begin{figure}[htp]
	
	\begin{tikzpicture}[scale=.6]
	
	\draw[thick] (0,0) circle (4cm);
	\draw[thick] (-3.8,0)--(3.8,0);
	\draw[thick] (0,-3.8)--(0,3.8);
	\draw[thick] (1,3.3)--(1,-3.3);
	\draw[thick] (3.3,1)--(-3.3,1);
	\draw[red,thick] (-1.2*1.3,2.4*1.3)--(1.2*1.3,-2.4*1.3);
	\draw[red,thick] (-2*1.7,-.7*1.7)--(2*1.7,.7*1.7);
	\draw[->,red,thick] (-1.6,2.6) to [out=200, in=70] (-2.7,1.3);
	\draw[->,red,thick] (-2.8,-1.3) to [out=300, in=180] (-1.5,-2.3);
	\node[red] at (-2,1.8) {$\alpha$};
	\node[red] at (-2,-1.7) {$\beta$};
	
	\end{tikzpicture}
	\caption{A projective picture emphasizing the moduli in Example~\ref{ex:TotallyNonFree}}\label{fig:TotallyNonFree}
\end{figure}
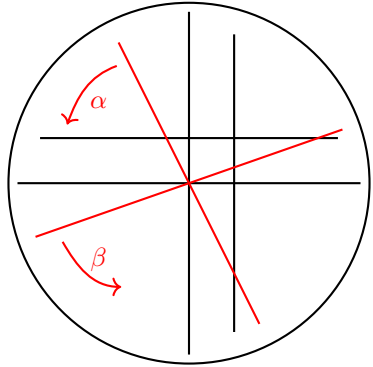
\end{exm}

\begin{exm}\label{ex:RestrictionHighPdim}
Let $S=\kk[x_0,\ldots,x_r]$ and let $\A\subset\kk^{r+1}$ be the arrangement defined by
\[
\Q(\A)=x_0\left(\prod\limits_{i=1}^r (x^2_i-x^2_0) \right) (x_1-x_2)\cdots (x_{r-1}-x_r)(x_r+x_1)
\]
Let $H$ be the hyperplane defined by $x_0$.  In Proposition~\ref{prop:Xr}, we will show that $\A$ is free using Yoshinaga's theorem~\cite{YoshCharacterizationFreeArr} and Theorem~\ref{thm:FreeMultNonFreeTF2}.  Moreover, we will prove that $\mbox{pdim}(D(\A^H))=r-3$, the largest possible.  In fact, we will show more: the minimal free resolution of $D(\A^H)$ is a truncated and shifted Koszul complex, so it is linear.  As with the previous two examples, the key to our analysis is that the restriction $\A^H$ is a $TF_2$ arrangement, which is particularly well suited to the homological methods we introduce in this paper.

This family of examples is interesting because it adds to a short list of arrangements known to fail Orlik's conjecture.  This conjecture states that $\A^H$ is free whenever $\A$ is free~\cite{OrlikArr}.  The only counterexamples to this conjecture of which we are aware appear in work of Edelman and Reiner~\cite{ReinerCounterEx,ReinerAnBn}.  For the small ranks that we have been able to compute, our examples differ from theirs in that $D(\A^H)$ for the examples of Edelman and Reiner seems to be always `almost free' - that is $D(\A^H)$ has only one more generator than the rank of $\A^H$ and there is only a single relation among these generators.  This latter behavior has been studied in a recent article of Abe~\cite{AbeDeletion}.
\end{exm}

\section{Preliminaries}\label{sec:Preliminaries}
Fix a field $\kk$, let $V$ be a $\kk$-vector space of dimension $\ell$, and $V^*$ the dual vector space.  Set $S=\mbox{Sym}(V^*)$, the symmetric algebra on $V^*$.  A hyperplane arrangement $\A\subset V$ is a union of hyperplanes $H$ defined by the vanishing of the affine linear form $\alpha_H\in V^*$; the \textit{defining polynomial} of $\A$ is $\Q(\A)=\prod_{H\in\A} \alpha_H$.  We will consistently abuse notation and write $H\in\A$ if $H$ is one of the hyperplanes whose union forms $\A$.  Moreover, we will write $|\A|$ for the number of hyperplanes in $\A$.

The \textit{rank} of a hyperplane arrangement $\A\subset V$ is $r=r(\A):=\dim V-\dim(\cap_{H\in \A} H)$.  The arrangement $\A\subset V$ is called \textit{essential} if $r(\A)=\dim V$ and \textit{central} if $\cap_i H_i\neq\emptyset$.  We will always assume $\A$ is a central hyperplane arrangement.  We refer the reader to the landmark book of Orlik and Terao~\cite{Arrangements} for further details on arrangements.

The intersection lattice $L=L(\A)$ of $\A$ is the lattice whose elements (flats) are all possible intersections of the hyperplanes of $\A$, ordered with respect to reverse inclusion.  We will use $<$ to denote the ordering on the lattice, so if $X,Y\in L(\A)$ and $X\subseteq Y$ as intersections, then $Y\le X$ in $L(\A)$.  This is a ranked lattice with rank function the codimension of the flat; we denote by $L_i=L_i(\A)$ the flats $X\in L(\A)$ with rank $i$.  Given a flat $X\in L(\A)$, the (closed) subarrangement $\A_X$ is the hyperplane arrangement of those hyperplanes of $\A$ which contain $X$, and the \textit{restriction} of $\A$ to $X$, denoted $\A^X$, is the hyperplane arrangement (in linear space corresponding to $X$) with hyperplanes $\{H\cap X: H \not < X  \mbox{ in } L(\A)\}$.  If $X<Y$, the interval $[X,Y]\subset L(\A)$ is the sub-lattice of all flats $Z\in L$ so that $X\le Z\le Y$.  This is the intersection lattice of the arrangement $\A^Y_X$.

If $\A\subset V_1$ and $\B\subset V_2$ are two arrangements, then the product of $\A$ and $\B$ is the arrangement
\[
\A \times \B=\{H\oplus V_2:H\in\A\}\cup\{V_1\oplus H':H'\in\B\},
\]
and the arrangements $\A,\B$ are \textit{factors} of $\A\times\B$.  If an arrangement can be written as a product of two arrangements we say it is \textit{reducible}, otherwise we call it \textit{irreducible}.  (Notice that an arrangement is not essential if and only if it has the empty arrangement as a factor).

If $\A\subset V$ is an arrangement the module of derivations of $\A$, denoted $D(\A)$, is defined by
\[
D(\A)=\{\theta\in\mbox{Der}_\kk(S)| \theta(\alpha_H)\in\langle \alpha_H \rangle\mbox{ for all } H\in\A \}.
\]
If $D(\A)$ is free as an $S$-module, we say $\A$ is free.
\begin{defn}
A multi-arrangement $(\A,\m)$ is an arrangement $\A\subset V$, along with a function $\m:\A\rightarrow \Z_{>0}$ assigning a positive integer to every hyperplane.  The \textit{defining polynomial} of a multi-arrangement $(\A,\m)$ is $\Q(\A,\m):=\prod_{H\in\A} \alpha_H^{\m(H)}$.  The module of multi-derivations $D(\A,\m)$ is
\[
D(\A,\m)=\{\theta\in\mbox{Der}_\kk(S)|\theta(\alpha_H)\in\langle \alpha_H^{\m(H)} \rangle\mbox{ for all } H\in\A\}
\]
\end{defn}
	
\begin{lem}\label{lem:DerivationMatrix}
Let $(\A,\m)$ be a multi-arrangement in $V\cong \kk^\ell$.  Let $\alpha_i$ be the form defining the hyperplane $H_i$, and set $m_i=\m(H_i)$.  The module $D(\A,\m)$ of multiderivations on $\A$ is isomorphic to the kernel of the map
\[
\psi:S^{\ell+d}\rightarrow S^d,
\]
where $\psi$ is the matrix
\[
\begin{pmatrix}
& \vline & \alpha_1^{m_1} & & \\
B & \vline  & & \ddots & \\
& \vline & & & \alpha_k^{m_k}
\end{pmatrix}
\]
and $B$ is the matrix with entry $B_{ij}=a_{ij}$, where $\alpha_j=\sum_{i,j} a_{ij} x_i$.
\end{lem}
\begin{proof}
See the comments preceding~\cite[Theorem~4.6]{DimSeries}.
\end{proof}

If $D(\A,\m)$ is free as an $S$-module then we say that the multi-arrangement $(\A,\m)$ is free and $\m$ is a \textit{free multiplicity} of $\A$.  If $D(\A,\m)$ is free there is (by definition) a \textit{basis} of derivations $\theta_1,\ldots,\theta_\ell\in D(\A,\m)$ so that every other $\theta\in D(\A,\m)$ can be written uniquely as a polynomial combination of $\theta_1,\ldots,\theta_\ell$.  If $\A$ is central (which we will assume throughout), we may assume these derivations are homogeneous with degrees $d_i=\deg(\theta_i)$.  The set $(d_1,\ldots,d_\ell)$ are called the \textit{exponents} of $D(\A,\m)$.  We will always assume $d_1\ge d_2\ge\cdots\ge d_\ell$.  Write $|\m|$ for $\sum_{H\in\A}\m(H)$.  It follows from Saito's criterion (below) that if $D(\A,\m)$ is free with exponents $(d_1,\ldots,d_\ell)$ then $\sum_{i=1}^\ell d_i=|\m|$.

\begin{prop}[Saito's criterion]\label{prop:Saito}
Let $(\A,\m)$ be a central arrangement in a vector space $V$ of dimension $\ell$, and write $\kk[x_1,\ldots,x_\ell]$ for $\mbox{Sym}(V^*)$.  Suppose $\theta_1,\ldots,\theta_\ell$ are derivations with $\theta_i=\sum_{j=1}^\ell \theta_{ij}\frac{\partial}{\partial x_i}$.  Write $M=M(\theta_1,\ldots,\theta_\ell)$ for the $\ell\times\ell$ matrix of coefficients $M_{ij}=\theta_{ij}$.  Then $D(\A,\m)$ is free with basis $\theta_1,\ldots,\theta_\ell$ if and only if $\det(M)$ is a scalar multiple of the defining polynomial $\Q(\A,\m)$.
\end{prop}

If $X\in L(\A)$, we write $(\A_X,\m_X)$ for the multi-arrangement $\A_X$ with multiplicity function $\m_X=\m|_{\A_X}$.  If $(\A_X,\m_X)$ is free for every $X\neq \cap_{H\in\A} H \in L$, then we say $(\A,\m)$ is \textit{locally free}; equivalently the associated sheaf $\widetilde{D(\A,\m)}$ is a vector bundle on $\mathbb{P}^{\ell-1}$.

\begin{prop}\cite[Proposition~1.7]{AbeSignedEliminable}\label{prop:pdimLB}
Let $(\A,\m)$ be a multi-arrangement, $X\in L(\A)$, and $(\A_X,\m_X)$ the corresponding closed subarrangement with restricted multiplicities.  Then $\mbox{pdim}(D(\A,\m))\ge \mbox{pdim}(D(\A_X,\m_X))$.
\end{prop}
					
\begin{lem}[Ziegler~\cite{ZieglerMulti}]\label{lem:globalUB}
For any arrangement $\A\subset V,\mbox{pdim}(D(\A,\m))\le r(\A)-2$.  In particular, if $r(\A)\le 2$ then $(\A,\m)$ is free.
\end{lem}

If $\A$ is an arrangement and $H\in\A$, we denote by $(\A^H,\m)$ the \textit{Ziegler restriction} of $\A$ to $\A^H$; this is the arrangement $\A^H$ with the multiplicity function $\m^H$ defined by
\[
\m^H(X)=\#\{H'\in\A:H'\cap H=X\}
\]
for every $X\in \A^H$.  We include the following criterion for freeness which is due to Yoshinaga~\cite{YoshCharacterizationFreeArr}; the observation that we can restrict to codimension three was made in~\cite[Theorem~4.1]{AbeYoshinaga}.

\begin{thm}\cite[Theorem~2.2]{YoshCharacterizationFreeArr}\label{thm:Yoshinaga}
An arrangement $\A$ over a field of characteristic zero is free if and only if, for some $H\in\A$:
\begin{enumerate}
\item $(\A^H,\m^H)$ is free and
\item $\A_X$ is free for every $X\neq 0\in L_3(\A)$ so that $H<X$.
\end{enumerate}
\end{thm}

The second condition is sometimes stated as `$\A$ is locally free along $H$ in codimension three.'

\section{The homological criterion}\label{sec:ChainComplex}
Let $(\A,\m)$ be a multi-arrangement.  In this section we prove Theorem~\ref{thm:Free}; we describe the chain complex $\cD^\bullet(\A,\m)$ and prove that $(\A,\m)$ is free if and only if $H^i(\cD^\bullet(\A,\m))=0$ for all $i>0$.  The construction of the modules which comprise $\cD^\bullet(\A,\m)$ is due to Brandt and Terao if $\m\equiv 1$~\cite{TeraoPoincare,BrandtTerao}; we make the straightforward observation that the same definitions work also for multi-arrangements.  We follow the presentation given in~\cite{BrandtTerao}.


\begin{defn}\label{defn:Dk}
Set $D_0(\A,\m)=D(\A,\m)$ and for $1\le k\le r=r(\A)$ inductively define $D_k(\A,\m)$ and $K_k(\A,\m)$ as the cokernel and kernel, respectively of the map
\[
\tau_{k-1}=\tau_{k-1}(\A): D_{k-1}(\A,\m)\rightarrow\bigoplus\limits_{X\in L_{k-1}} D_{k-1}(\A_X,\m_X),
\]
where $\tau_{k}$ is a sum of maps $\phi_{k}(Y):D_{k}(\A,\m)\rightarrow D_{k}(\A_Y,\m_Y)$.  For $Y\in L$ with $r(Y)\ge k$, $\phi_k(Y)$ is defined inductively (the map for $k=0$ is the usual inclusion of derivations) via the diagram in Figure~\ref{fig:Dk}:
\begin{figure}[htp]
\centering
\begin{tikzcd}
D_{k-1}(\A,\m) \ar{r}{\tau_{k-1}(\A)} \ar{d}{\phi_{k-1}(Y)} & \bigoplus\limits_{X\in L_{k-1}} D_{k-1}(\A_X,\m_X) \ar{r}\ar{d}{p_{k-1}(Y)} & D_k(\A,\m) \ar{r}\ar{d}{\phi_k(Y)} & 0\\
D_{k-1}(\A_Y,\m_Y) \ar{r}{\tau_{k-1}(\A_Y)} & \bigoplus\limits_{\substack{X\le Y\\ r(X)=k-1}} D_{k-1}((\A_Y)_X,(\m_Y)_X) \ar{r} & D_k(\A_Y,\m_Y) \ar{r} & 0
\end{tikzcd}
\caption{Diagram for Definition~\ref{defn:Dk}}\label{fig:Dk}
\end{figure}
The center vertical map is projection, the left-hand square commutes, so $\phi_k(Y)$ may be defined so that the right-hand square commutes.
\end{defn}

\begin{remark}\label{rem:D1}
Given an arrangement $\A$, the only flat of $L$ with rank $0$ is $V$, the ambient space of $\A$.  The module $D_1(\A,\m)$ is the cokernel of the map
\[
D_0(\A,\m)\xrightarrow{\tau_0} \bigoplus\limits_{X\in L_0} D_0(\A_X,\m),
\]
in other words the cokernel of the inclusion
\[
D(\A,\m)\rightarrow D(V)=\mbox{Der}_\kk(S)\cong S^\ell,
\]
where $\ell=\dim(V)$.
\end{remark}

\begin{remark}\label{rem:LowDescription}
Fix a basis $x_1,\ldots,x_\ell$ for $S_1=\mbox{Sym}(V^*)_1$ and denote the corresponding basis of $\mbox{Der}_{\kk}(S)$ by $\partial_i=\partial/\partial x_i$.  Number the hyperplanes of $\A$ by $H_1,\ldots,H_k$.  Assume $H_j=V(\alpha_j)$, where $\alpha_j=\alpha_{H_j}=\sum_{i} a_{ij}x_i$.  For some $H=H_j\in\A$ let $\partial_H=\sum_i a_{ij}\partial_i$.

For $H\in \A$, let $J(H)=\langle \alpha_H^{\m(H)} \rangle$, the ideal generated by $\alpha_H^{\m(H)}$ in $S$.  Then $D(\A_H,\m_H)\subset \mbox{Der}_\kk(S)$ is isomorphic to $J(H)\partial_H\oplus S^{\ell-1}$, where $\ell=\dim(V)$ and $J(H)\partial_H$ denotes that $J(H)$ is living inside of the copy of $S$ corresponding to the basis element $\partial_H$.  So $D_1(\A_H,\m_H)$ is the cokernel of the inclusion $D(\A_H,\m_H)\rightarrow \mbox{Der}_{\kk}(S)\cong S^\ell,$
which may be identified as $S\partial_H/J(H)$.  There is then a natural map
\[
\mbox{Der}_{\kk}(S)\cong S^\ell\xrightarrow{B} \bigoplus_{H\in\A} \dfrac{S}{J(H)}=\bigoplus_{X\in L_1} D_1(\A_X,\m_X),
\]
where $B$ is the matrix with entries $B_{ij}=a_{ij}$.  The kernel of this map is $D(\A,\m)$, its image is $D_1(\A,\m)$, and its cokernel is $D_2(\A,\m)$.
\end{remark}

\begin{remark}
We will discuss computations of $D_k(\A,\m)$ further in \S~\ref{sec:Computations}.
\end{remark}

Extending Remark~\ref{rem:LowDescription}, we assemble the modules $\bigoplus\limits_{X\in L_k} D_k(\A_X,\m_X)$ into a chain complex.

\begin{defn}\label{defn:DerivationComplex}
Set $\cD^k=\bigoplus\limits_{X\in L_k} D_k(\A_X,\m_X)$.  Define $\delta^k:\cD^k\rightarrow \cD^{k+1}$ by the composition $\cD^k\rightarrow D_{k+1}(\A,\m) \xrightarrow{\tau_{k+1}} \cD^{k+1},$ where the first map is the natural surjection from Definition~\ref{defn:Dk}.  The derivation complex $\cD^\bullet=\cD^{\bullet}(\A,\m)$ is the chain complex with modules $\cD^k$ for $k=0,\ldots,r(\A)$ and maps $\delta^k:\cD^k\rightarrow \cD^{k+1}$ for $k=0,\ldots,r(\A)-1$.
\end{defn}

\begin{remark}\label{rem:ComplexDiagram}
The derivation complex $\cD^\bullet$ is tautologically a complex from the definitions of $D_k(\A,\m)$ and $\delta^k$.  The commutative diagram in Figure~\ref{fig:DerComplex} shows how all the definitions so far fit together.  Note that $K_i(\A,\m)$ from Definition~\ref{defn:Dk} may be identified with $H^i(\cD^\bullet)$.
\end{remark}

\begin{remark}
The chain complex $\cD^\bullet$ in Definition~\ref{defn:DerivationComplex} is essentially dual to a chain complex described in~\cite{StefanFormal}; we will describe the precise connection in \S~\ref{sec:Formal}.
\end{remark}

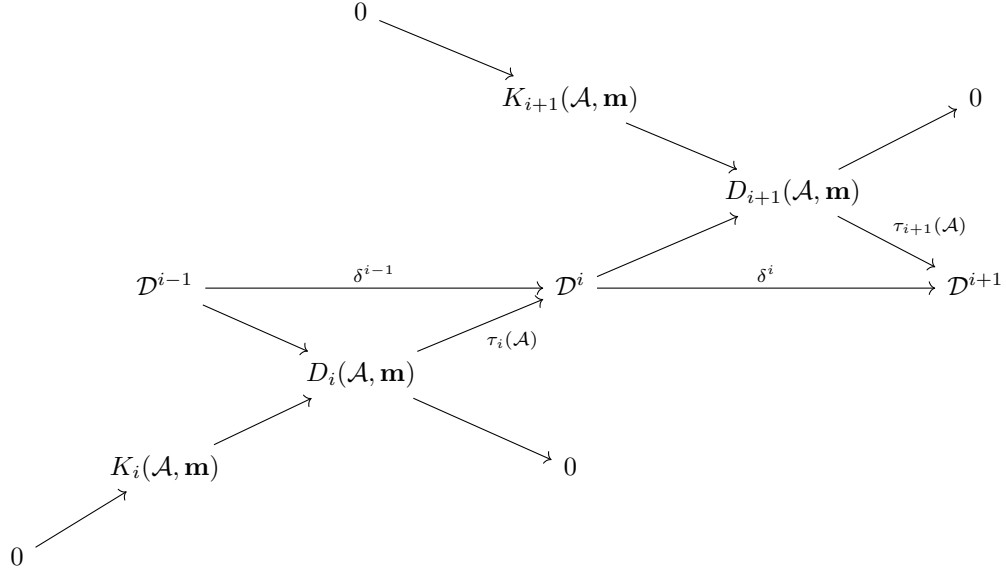
\begin{figure}	
\centering
\begin{tikzcd}
& & 0\ar[dr] & & & \\
& & & K_{i+1}(\A,\m)\ar{dr} & & 0 \\
& & & & D_{i+1}(\A,\m)\ar{ur}\ar{dr}{\tau_{i+1}(\A)} \\
& \cD^{i-1} \ar{dr}\ar{rr}{\delta^{i-1}} & & \cD^i\ar{ur}\ar{rr}{\delta^i} & & \cD^{i+1} \\
& & D_i(\A,\m) \ar{ur}[swap]{\tau_i(\A)}\ar[dr] & & & \\
& K_i(\A,\m)\ar[ur] & & 0 & &\\
0\ar[ur] & & & & &
\end{tikzcd}
\caption{Components of Definition~\ref{defn:DerivationComplex}}\label{fig:DerComplex}
\end{figure}

\begin{lem}\label{lem:H0Der}
For a multi-arrangement $(\A,\m)$, $H^0(\cD^\bullet(\A,\m))\cong D(\A,\m)$.
\end{lem}
\begin{proof}
This is immediate from Remark~\ref{rem:LowDescription}.
\end{proof}

Now we proceed to the proof of Theorem~\ref{thm:Free}.  We use a few preliminary results.

\begin{lem}\cite[Lemma~4.12]{BrandtTerao}\label{lem:local}
For any $k$, the functors $X\to D_k(\A_X,\m_X)$ for $X\in L$ are \textit{local} in the sense of~\cite[Definition~6.4]{SolomonTeraoCharPoly}.
Namely let $P\in \mbox{Spec}(S)$, $X\in L$, and set $X(P)=\bigcap\limits_{\substack{H\in\A_X\\ \alpha_H\in P}} H$.  Then 
\begin{itemize}
\item $D_k(\A_X,\m_X)_P=D_k(\A_{X(P)},\m_{X(P)})_P$ and
\item $\cD^\bullet(\A,\m)_P=\cD^\bullet(\A_{X(P)},\m_{X(P)})_P$.
\end{itemize}
\end{lem}

\begin{proof}
For the first bullet, use the fact that $X\to D(\A_X,\m)$ is local, the short exact sequences in Definition~\ref{defn:Dk}, and the fact that localization is an exact functor.  The second bullet follows from the first.
\end{proof}

\begin{prop}\label{prop:CMCodimK}
Let $X\in L_k$ and $I(X)\subset S$ denote the ideal generated by the linear forms $\alpha_H$ for all $H\le X$.  Then $D_k(\A_X,\m_X)$ is Cohen-Macaulay of codimension $k$ and $I(X)$ is its only associated prime.
\end{prop}
\begin{remark}
Proposition~\ref{prop:CMCodimK} is implicit in the proof of~\cite[Proposition~4.13]{BrandtTerao}; we provide a proof for completeness.
\end{remark}
\begin{proof}
As usual, set $\ell=\dim(V)$.  By changing coordinates, we may assume $X=V(x_1,\ldots,x_k)$.  The result is clear if $k=0$ or $k=1$, so we assume $k\ge 2$.  Let $\pi_X:V\to X^\perp=W$ be the projection with center $X$ and set $R=\mbox{Sym}(W^*)\cong \kk[x_{k+1},\ldots,x_\ell]$.  Then we observe that
\begin{itemize}
\item $\A^\pi=\pi_X(\A_X)$ is an essential arrangement in $W$ of rank $\ell-k=\dim W$,
\item $D_k(\A^\pi,\m_X)\otimes_R S=D_k(\A_X,\m_X)$,
\item $x_{k+1},\ldots,x_{\ell}$ is a regular sequence on $D_k(\A_X,\m_X)$,
\item $D_k(\A_X,\m_X)/\langle x_{k+1},\ldots,x_{\ell} \rangle D_k(\A_X,\m_X)\cong D_k(\A^\pi,\m_X)$,
\item and $\mbox{Ass}(D_k(\A_X,\m_X))=\{PS|P\in\mbox{Ass}(D_k(\A^\pi,\m_X))\},$
\end{itemize}
where the final bullet point follows from~\cite[Theorem~23.2]{Matsumura}, which describes behavior of associated primes under flat extensions.  Hence it suffices to show that the only associated prime of $D_k(\A,\m)$ when $k=r(\A)=\dim V$ is the maximal ideal of $S$.  Consider the short exact sequence
\[
0\rightarrow D_{k-1}(\A,\m)\rightarrow \cD^{k-1}=\bigoplus\limits_{X\in L_{k-1}} D_{k-1}(\A_X,\m_X) \rightarrow D_k(\A,\m) \rightarrow 0
\]
from Definition~\ref{defn:Dk}, and localize at a prime $P\in\mbox{Spec}(S)$.  If $\mbox{codim}(P)\le k-1$, then by induction either $\cD^{k-1}_P$ vanishes (in which case $D_k(\A,\m)_P=0$) or $P=I(X)$ for some $X\in L$ of codimension $k-1$ and $\cD^{k-1}_P=D_{k-1}(\A_X,\m_X)_P$.  In the latter case, localizing the exact sequence above at $P=I(X)$ and using Lemma~\ref{lem:local} yields the exact sequence
\[
0\rightarrow D_{k-1}(\A_X,\m_{X})_{I(X)}\rightarrow D_{k-1}(\A_X,\m_X)_{I(X)} \rightarrow D_k(\A,\m)_{I(X)} \rightarrow 0,
\]
so clearly $D_k(\A,\m)_{I(X)}=0$.  Hence the only prime in the support of $D_k(\A,\m)$ is the homogeneous maximal ideal.
\end{proof}

\begin{proof}[Proof of Theorem~\ref{thm:Free}]
By Lemma~\ref{lem:H0Der}, $D(\A,\m)\cong H^0(\cD^\bullet(\A,\m))$.  Now we use the following result of Schenck and Stiller (see also~\cite{Spect}).
\begin{thm}\cite[Theorem~3.4]{CohVan}\label{thm:Schenck}
Suppose $C^\bullet=0\rightarrow C^0 \rightarrow C^1 \rightarrow C^2 \rightarrow \cdots\rightarrow C^t \rightarrow 0$ is a complex of $S=\kk[x_1,\ldots,x_\ell]$-modules so that, for $k=0,\ldots t$,
\begin{itemize}
\item $C^k$ is Cohen-Macaulay of codimension $k$
\item $H^k(\mathcal{C^\bullet})$ is supported in codimension $\ge k+2$.
\end{itemize}
Then $H^0(C^\bullet)$ is free if and only if $H^k(C^\bullet)=0$ for $k>0$ and locally free if and only if $H^k(C^\bullet)$ has finite length for $k>0$.
\end{thm}
\noindent By Proposition~\ref{prop:CMCodimK}, $\cD^k=\cD^k(\A,\m)$ is Cohen-Macaulay of codimension $k$.  So we need to show that $H^k(\cD^\bullet)$ is supported in codimension at least $k+2$.  We use the fact that taking homology commutes with localization.  So let $P$ be a prime and consider the localized complex
\[
\cD^\bullet(\A,\m)_P=\cdots\rightarrow \cD^{k-1}_P \xrightarrow{\delta^{k-1}_P} \cD^k_P \xrightarrow{\delta^k_P} \cD^{k+1}_P \rightarrow\cdots
\]
If $\mbox{codim}(P)\le k$, then we have seen in the proof of Proposition~\ref{prop:CMCodimK} that the localized map $\delta^{k-1}_P$ becomes an isomorphism, hence $H^k(\cD^\bullet)_P=H^k(\cD^\bullet_P)=0$.  Now suppose $\mbox{codim}(P)=k+1$.  If $P\neq I(X)$ for some $X\in L$ of codimension $k+1$, then let $X\in L_i$ ($i\le k$) be the flat of maximal rank so that $I(X)\subset P$.  If $r(X)\le k-1$ then $H^k(\cD^\bullet_P)=0$ by Proposition~\ref{prop:CMCodimK}.  So suppose $X$ has codimension $k$.  Then the localized map $\delta^{k-1}_P$ becomes an isomorphism again as in the proof of Proposition~\ref{prop:CMCodimK}.

Finally suppose $P=I(X)$ for some $X\in L_{k+1}$.  Localizing yields
\[
\bigoplus\limits_{\substack{Y\ge X \\ r(Y)=k-1}} D_{k-1}(\A_Y,\m_Y)_{P} \xrightarrow{\delta_P^{k-1}} \bigoplus\limits_{\substack{Z\ge X\\ r(Z)=k}} D_k(\A_Z,\m_Z)_{P} \xrightarrow{\delta^k_P} D_{k+1}(\A_X,\m_X)_{P}.
\]
By definition $\delta^{k-1}$ factors through $D_k(\A,\m)$.  Hence $H^k(\cD^\bullet)_P$ is the middle homology of the three term complex
\[
0\rightarrow D_k(\A_X,\m_X)_P \xrightarrow{(\tau_k)_P} \bigoplus\limits_{\substack{Z\ge X\\ r(Z)=k}} D_k(\A_Z,\m_Z)_{P} \xrightarrow{\delta^k_P} D_{k+1}(\A_X,\m_X)_{P}\rightarrow 0,
\]
which is exact by Definition~\ref{defn:Dk}.  It follows that $H^k(\cD^\bullet)$ is supported in codimension $\ge k+2$.
\end{proof}

\begin{remark}
In the case of a simple arrangement, the forward implication of Theorem~\ref{thm:Free} follows from~\cite[Proposition~4.13]{BrandtTerao}.
\end{remark}



Theorem~\ref{thm:Schenck} arises from a studying the hyperExt modules of $\cD^\bullet(\A,\m)$.  Without the vanishing assumptions we may obtain the following.

\begin{prop}\label{prop:BoundingProjectiveDimension}
Set $p_i=\mbox{pdim}(H^i(\cD^\bullet(\A,\m)))$ for $i>0$.  Then 
\[
\mbox{pdim}(D(\A,\m))\le \max\limits_{i>0}\{p_i-i-1\},
\]
with equality if there is a single $i>0$ for which $H^i(\cD^\bullet(\A,\m))\neq 0$.
\end{prop}
\begin{proof}
See~\cite[Lemma~4.11]{Spect} or~\cite[\S~3]{GSplinesGraphicArr}.
\end{proof}

\subsection{A combinatorial bound on projective dimension}

We close this section by extending a combinatorial bound on projective dimension due to Kung and Schenck for simple arrangements~\cite[Corollary~2.3]{HalKung}.  Recall that a generic arrangement of rank $\ell$ is one in which the intersection of every subset of $k\le\ell$ hyperplanes has codimension $k$.

\begin{cor}\label{cor:pdimcircuit}
Let $(\A,\m)$ be a multi-arrangement.  If $\A_X$ is generic with $|\A_X|>r(X)$, then $\mbox{pdim}(D(\A,\m))\ge r(X)-2$.  In particular, if the matroid of $\A$ has a closed circuit of length $m$, then $\mbox{pdim}(D(\A,\m))\ge m-3$.
\end{cor}
\begin{proof}
If $r(\A)=2$ the statement is trivial so we will assume $r(\A)>2$.  Suppose $\A_X$ is generic with $|\A_X|>r(X)$.  By Proposition~\ref{prop:pdimLB}, it suffices to show that $\mbox{pdim}(D(\A_X,\m_X))\ge r(X)-2$.  So we assume $\A=\A_X$ is essential and generic of rank $r$ with $|\A|>r$ and prove $\mbox{pdim}(D(\A,\m))=r-2$.

In this case we claim the chain complex $\cD^\bullet(\A,\m)$ has the form $S^r \xrightarrow{\delta^0} \bigoplus\limits_{H\in\A}\dfrac{S}{J(H)}$, where $J(H)=\langle \alpha_H^{\m(H)}\rangle$.  That $\cD^0=S^r$ and $\cD^1=\bigoplus_{H\in\A} S/J(H)$ follows from the definition of $\cD^\bullet$ and Remark~\ref{rem:LowDescription}.  To prove that $\cD^k=0$ for $k>1$, it suffices to show that $D_2(\A_Y,\m_Y)=0$ for all $Y\in L_2$.  We have
\[
D_2(\A_Y,\m_Y)=\mbox{coker}\left( S^{r}\xrightarrow{\delta^0_Y} \bigoplus_{H\in\A_Y}\dfrac{S}{J(H)}\right).
\]
Since $\A$ is generic, the set $\{\alpha_H:H\in\A_Y\}$ consists of $r(Y)$ linearly independent forms and the coefficient matrix $\delta^1_Y$ has full rank.  So $D_2(\A_Y,\m_Y)=0$.

It follows that $H^1(\cD^\bullet(\A,\m))=\mbox{coker}(\delta^0)$.  Since $|\A|>r$, we see that $\delta^0$ cannot be surjective, so $H^1(\cD^\bullet(\A,\m))\neq 0$.  We show that $H^1(\cD^\bullet)$ is only supported at the maximal ideal.  To this end, let $P\in\mbox{spec}(S)$ be a prime of codimension $k\le r-1$.  Write $X(P)=\bigcap\limits_{\substack{H\in\A_X\\ \alpha_H\in P}} H$.  Since $\A$ is generic, $\{\alpha_H: \alpha_H\in P\}$ consists of at most $k$ linearly independent forms, so up to a change of coordinates $\A_{X(P)}$ is union of coordinate hyperplanes.  By Lemma~\ref{lem:local}, $\cD^\bullet(\A,\m)_P\cong \cD^\bullet(\A_{X(P)},\m_{X(P)})_P$.  The chain complex $\cD^\bullet(\A_{X(P)},\m_{X(P)})$ has the form $S^r\xrightarrow{\delta^0_{X(P)}} \bigoplus\limits_{H\in\A_{X(P)}} \dfrac{S}{J(H)}$, and $\delta^0_{X(P)}$ is clearly surjective, so
\[
H^1(\cD^\bullet(\A,\m))_P\cong H^1(\cD^\bullet(\A,\m)_P)\cong H^1(\cD^\bullet(\A_{X(P)},\m_{X(P)})_P)=0.
\]
It follows that $H^1(\cD^\bullet(\A,\m))$ is only supported at the maximal ideal.  Since $H^1(\cD^\bullet(\A,\m))\neq 0$, $\mbox{pdim}(H^1(\cD^\bullet))=r$ and by Proposition~\ref{prop:BoundingProjectiveDimension}, $\mbox{pdim}(D(\A,\m))=r-2$, the maximal projective dimension.
\end{proof}

\begin{remark}
Corollary~\ref{cor:pdimcircuit} implies that generic arrangements are totally non-free; this was first proved by Yoshinaga~\cite{YoshExtendability}.
\end{remark}

\begin{remark}
Even for simple arrangements, the lower bound given by Corollary~\ref{cor:pdimcircuit} may be arbitrarily far off from the actual projective dimension.  See Remark~\ref{rem:GeneralizedX3}.
\end{remark}

\section{Multi-arrangements and $k$-formality}\label{sec:Formal}
In this section we will show that if $(\A,\m)$ is a free multi-arrangement then $\A$ is $k$-formal (in the sense of~\cite{BrandtTerao}) for $2\le k\le r-1$, where $r=r(\A)$ is the rank of $\A$ (thus generalizing the result of Brandt and Terao~\cite{BrandtTerao} to multi-arrangements).  Once we have set up the notation, this is an immediate corollary of Theorem~\ref{thm:Free}.

We again follow the presentation in~\cite{BrandtTerao}.  Fix an arrangement $\A=\cup_{H\in\A} V(\alpha_H)\subset V$.  Set $E(\A):=\bigoplus_{H\in\A} e_H\kk$ and define $\phi:E(\A)\rightarrow V^*$ by $\phi(e_H)=\alpha_H$.  Put $F(\A)=\ker(\phi)$; this is called the \textit{relation space} of $\A$.

The arrangement $\A$ is $2$-\textit{formal} (or just \textit{formal}) if the relation space is generated by relations among three linear forms.  Since three linear forms are dependent if and only if they define a codimension two flat, $2$-formality is equivalent to surjectivity of the map
\[
\pi_2:\bigoplus_{X\in L_2} F(\A_X)\rightarrow F(\A),
\]
where $\pi_2$ is the sum of natural inclusions $F(\A_X)\hookrightarrow F(\A)$ for each $X\in L_2$.

\begin{defn}\label{defn:Rk}
Set $R_0:=T(\A)^*\subset V^*$, where $T(\A)=\cap_{H\in\A} H$.  For $1\le k\le r$, recursively define $R_k(\A)$ as the kernel of the map
\[
\pi_{k-1}=\pi_{k-1}(\A):=\bigoplus_{X\in L_{k-1}} R_{k-1}(\A_X)\rightarrow R_{k-1}(\A),
\]
where $\pi_k$ is the sum of natural inclusions for $0\le k\le r-1$.  To simplify notation, set $\cR_k=\cR_k(\A)=\bigoplus_{X\in L_k} R_k(\A_X)$.
\end{defn}

\begin{remark}
After chasing through the definitions one can see that $R_1(\A)$ is the kernel of the restriction map $V^*\rightarrow T(\A)^*$ and $R_2(\A)=F(\A)$.  See~\cite{BrandtTerao} for details.
\end{remark}

\begin{defn}
The arrangement $\A$ is
\begin{itemize}
\item $2$-formal if $\A$ is formal
\item $k$-formal, for $3\le k\le r-1$, if $\A$ is $(k-1)$-formal and the map $\pi_k:\cR_k=\bigoplus_{X\in L_k} R_k(\A_X)\rightarrow R_k(\A)$ is surjective.
\end{itemize}
\end{defn}

In~\cite{StefanFormal}, Tohaneanu gives a homological formulation of $k$-formality as follows.  First, notice that there is a natural differential $\delta_k: \cR_k\rightarrow \cR_{k-1}$ (similar to the differential for $\cD^\bullet$) defined as the composition $\cR_k \rightarrow R_k(\A) \xrightarrow{\pi_{k-1}} \cR_{k-1}$.

\begin{lem}\cite[Lemma~2.5]{StefanFormal}\label{lem:HomologicalCharFormality}
With the differentials $\delta_k$, $1\le k\le r$, the vector spaces $\cR_i$ ($0\le k\le r$) form a chain complex $\cR_\bullet=\cR_\bullet(\A)$.  The arrangement $\A$ is $k$-formal if and only if $H_i(\cR_\bullet)=0$ for $i=1,\ldots,k-1$.
\end{lem}

\begin{remark}
If $\m\equiv 1$  (so $(\A,\m)$ is a simple arrangement) we will denote $D_k(\A,\m)$ (recall Definition~\ref{defn:Dk}) and $\cD^\bullet(\A,\m)$ (recall Definition~\ref{defn:DerivationComplex}) by $D_k(\A)$ and $\cD^\bullet(\A)$, respectively.
\end{remark}

Brandt and Terao show that the vector spaces $R_k(\A)$ are dual to the degree zero part of $D_k(\A)$.

\begin{prop}\cite[Proposition~4.10]{BrandtTerao}\label{prop:Duality}
For $0\le k\le r$, $D_k(\A)_0\cong R_k(\A)^*$, where $R_k(\A)^*$ is the $\kk$-vector space dual of $R_k(\A)$.
\end{prop}

\begin{lem}\label{lem:DegZeroGen}
The modules $D_k(\A,\m)$ for $1\le k\le r$ are generated in degree zero.  More precisely, we have an isomorphism (as $\kk$-vector spaces) $D_k(\A,\m)_0\cong D_k(\A)_0$.
\end{lem}
\begin{proof}
Both claims are clear for $D_1(\A,\m)$ by Remark~\ref{rem:D1}.  By Definition~\ref{defn:Dk}, $D_k(\A,\m)$ is a quotient of $\bigoplus_{X\in L_k} D_{k-1}(\A_X,\m_X)$.  Hence by induction, $D_k(\A,\m)$ is also generated in degree zero.  Now we have the following commutative diagram:
\begin{center}
\begin{tikzcd}
D_{k-1}(\A,\m)_0 \ar{r}{\tau_{k-1}(\A)} \ar{d}{\cong} & \bigoplus\limits_{X\in L_{k-1}} D_{k-1}(\A_X,\m_X)_0 \ar{r}\ar{d}{\cong} & D_k(\A,\m)_0 \ar{r} & 0\\
D_{k-1}(\A)_0 \ar{r}{\tau_{k-1}(\A)} & \bigoplus\limits_{X\in L_{k-1}} D_{k-1}(\A_X)_0 \ar{r} & D_k(\A)_0 \ar{r} & 0,
\end{tikzcd}
\end{center}
where the first two vertical maps are isomorphisms by induction.  Hence there is also an isomorphism $D_k(\A,\m)_0\cong D_k(\A)$.
\end{proof}

\begin{cor}\label{cor:HomologicalCharFormality}
An arrangement $\A$ is $k$-formal if and only if $H^i(\cD^\bullet(\A,\m)_0)=0$ for $i=1,\ldots,k-1$.
\end{cor}
\begin{proof}
Immediate from Lemma~\ref{lem:DegZeroGen}, Lemma~\ref{lem:HomologicalCharFormality}, and Proposition~\ref{prop:Duality}.
\end{proof}

\begin{defn}\label{defn:TotallyFormal}
An arrangement $\A$ is \textit{totally formal} if $\A_X$ is $k$-formal for $2\le k\le r(X)$ for all $X\in L(\A)$.
\end{defn}

For example, a rank three arrangement is totally formal if and only if it is formal.  See Remark~\ref{rem:GraphicTotallyFormal} for further examples of totally formal arrangements.

\begin{cor}\label{cor:MultifreeFormal}
If $(\A,\m)$ is free then $\A$ is totally formal.
\end{cor}
\begin{proof}
Suppose to the contrary that $\A_X$ is not $k$-formal for some $X\in L$ and $2\le k\le r(X)-1$.  Then, by Corollary~\ref{cor:HomologicalCharFormality}, $H^i(\cD^\bullet_0(\A_X,\m_X))\neq 0$ for some $1\le i\le k-1$.  Hence by Theorem~\ref{thm:Free}, $D(\A_X,\m_X)$ is not free, whence $D(\A,\m)$ is not free by Proposition~\ref{prop:pdimLB}.
\end{proof}

\begin{remark}
We will see in Proposition~\ref{prop:H2Pres} that there are totally formal arrangements which nevertheless are totally non-free.  See also Example~\ref{ex:ZieglerPair}.
\end{remark}

\begin{remark}\label{rem:CombFreeObstFromFormality}
The ranks of the vector spaces appearing in $\cR_\bullet$ are not combinatorial in general (see Example~\ref{ex:ZieglerPair}), however if $\A$ is totally formal then these ranks are determined by $L(\A)$.  We can see this by inductively reading off the rank of $R_k(\A_X)$ $(X\in L_k)$ from the Euler characteristic of $\cR_\bullet(\A_X)$; since $\A$ is totally formal the Euler characteristic of $\cR_\bullet(\A_X)$ is zero by Lemma~\ref{lem:HomologicalCharFormality}.  This yields a number of combinatorial obstructions to freeness which can be read off $L(\A)$ (see for instance~\cite[Corollary~4.16]{BrandtTerao}).  By Corollary~\ref{cor:MultifreeFormal}, if any of these combinatorial obstructions are satisfied, the arrangement is totally non-free.
\end{remark}


In the following corollary, we call a hyperplane $H\in\A$ \textit{generic} if, for all $X\in L_2$ so that $H<X$ in $L$, there is a unique hyperplane $H'\neq H$ so that $H'<X$.  Moreover, we say $H$ is a \textit{separator} of $\A$ if $r(\A-H)<r(\A)$.  Part of the following result may be found in~\cite[Proposition~3.9]{BrandtTerao}; we provide a proof for completeness.

\begin{cor}\label{cor:genericHyperplane}
Suppose $\A$ is an arrangement of rank $\ge 2$.  If $\A$ has a generic hyperplane which is not a separator, then $\A$ is not formal.  In particular, $\A$ is totally non-free.
\end{cor}
\begin{proof}
Let $H\in\A$ be the generic hyperplane which is not a separator, and write $v_H$ for the corresponding row of $\delta^0_S$.  The condition that $H$ is not a separator means that we can find $r=r(\A)$ linearly independent rows $v_1,\ldots,v_r$ of $\delta^0_S$ where $v_i\neq v_H$ for $i=1,\ldots,r$.  Hence there is a relation $\sum_{i=1}^r c_iv_i+c_Hv_H=0$ (for constants $c_1,\ldots,c_H$).  Since $r\ge 2$ and $H$ is generic, there is no way to write this relation as a linear combination of relations among three hyperplanes (since $v_H$ is not in the support of any such relation).  So $\A$ is not formal.  The final conclusion follows from Corollary~\ref{cor:MultifreeFormal}.
\end{proof}

%

\section{Computing the chain complex}\label{sec:Computations}

In this section we work out concrete presentations for the modules appearing in $\cD^\bullet(\A,\m)$ and illustrate the constructions via examples, with the goal of studying freeness and projective dimension of $D(\A,\m)$.  The following definition, which constructs $\cD^\bullet$ as the cokernel of a map of chain complexes, is analogous to the setup of the Billera-Schenck-Stillman chain complex used in algebraic spline theory~\cite{Homology,LCoho}.  Since there are many details, the reader may find it easiest to read the following constructions while following along with Examples~\ref{ex:PointsP1} and~\ref{ex:X3}.

\begin{defn}\label{defn:FormalityComplex}
For a multi-arrangement $(\A,\m)$, set $S_k(\A_X)=D_k(\A_X,\m_X)_0\otimes_{\kk}S$, the degree zero part of  $D_k(\A_X,\m_X)$ tensored with $S$, and set $\cS^\bullet(\A):=\cD^\bullet(\A,\m)_0\otimes_{\kk}S$, so $\cS^k=\bigoplus_{X\in L_k} S_k(\A_X)$.  These are independent of the choice of multiplicities by Lemma~\ref{lem:DegZeroGen}.

For $Y\in L$, write $\phi^S_k(Y),\tau^S_k$ for the maps $\phi^S_k(Y):S_k(\A)\rightarrow S_k(\A_Y),\tau^S_k:S_k(\A)\rightarrow \bigoplus_{X\in L_k} S_k(\A_X)$ which are obtained from the maps $\phi_k(Y):D_k(\A,\m)\rightarrow D_k(\A_Y,\m),\tau_k:D_k(\A,\m)\rightarrow \bigoplus_{X\in L_k} D(\A_X,\m_X)$ (see Definition~\ref{defn:Dk}) by restricting to degree zero and then tensoring with $S$.  Likewise write $\delta_S^i$ for the differential of $\cS^\bullet$.

Since each of the modules $D_k(\A,\m)$ is generated in degree zero by Lemma~\ref{lem:DegZeroGen}, there is a natural surjective map $S_k(\A_X)\rightarrow D_k(\A_X,\m_X)$ for every $\m$ and $X\in L_k$.  Hence there is a surjective map of complexes $\cS^\bullet(\A)\rightarrow \cD^\bullet(\A,\m)$ for any multiplicity $\m$.
	
For each surjection $S_k(\A_X)\rightarrow D_k(\A_X,\m_X)$, write $J_k(\A_X,\m_X)$ for the kernel of this surjection, and write $\cJ^\bullet(\A,\m)$ for the kernel of the surjection $\cS^\bullet(\A)\rightarrow \cD^\bullet(\A,\m)$, so $\cJ^k(\A,\m)=\bigoplus_{X\in L_k} J_k(\A_X,\m_X)$.  Denote by $\phi^J_i(Y),\tau^J_i,$ and $\delta_J^i$ the maps obtained from restricting $\phi^S_i(Y),\tau^S_i,$ and $\delta_S^i$.  See figure~\ref{fig:FormalityComplexes} which shows the short exact sequence of complexes $0\rightarrow \cJ^\bullet \rightarrow \cS^\bullet\rightarrow \cD^\bullet \rightarrow 0$.
\end{defn}

\begin{figure}
\centering
\begin{tikzcd}
\cJ^\bullet(\A,\m) \cdots \ar{r} & \bigoplus\limits_{X\in L_{k-1}} J_{k-1}(\A_X,\m_X) \ar{r}{\delta_J^{k-1}} \ar{d} & \bigoplus\limits_{Y\in L_k} J_k(\A_Y,\m) \ar{d}\ar{r}{\delta_J^k} &\cdots 	\\	
\cS^\bullet(\A) \cdots \ar{r} & \bigoplus\limits_{X\in L_{k-1}} S_{k-1}(\A_X) \ar{r}{\delta_S^{k-1}} \ar{d} & \bigoplus\limits_{Y\in L_k} S_k(\A_Y) \ar{d}\ar{r}{\delta_S^k} &\cdots \\
\cD^\bullet(\A,\m) \cdots \ar{r} & \bigoplus\limits_{X\in L_{k-1}} D_{k-1}(\A_X,\m_X) \ar{r}{\delta^{k-1}} & \bigoplus\limits_{Y\in L_k} D_k(\A_Y,\m) \ar{r}{\delta^k} &\cdots
\end{tikzcd}
\caption{Short exact sequence of complexes from Definition~\ref{defn:FormalityComplex}}\label{fig:FormalityComplexes}
\end{figure}

\begin{remark}\label{rem:FormalS}
By Corollary~\ref{cor:HomologicalCharFormality}, $\A$ is $k$-formal if and only if $H^i(\cS^\bullet(\A))=0$ for $1\le i\le k-1$.  Furthermore $\A$ is essential if and only if $H^0(\cS^\bullet(\A))=0$.
\end{remark}

\begin{remark}\label{rem:LES}
	The short exact sequence $0\rightarrow \cJ^\bullet\rightarrow \cS^\bullet \rightarrow \cD^\bullet\rightarrow 0$ gives rise to a long exact sequence starting as
	\[
	0\rightarrow H^0(\cS^\bullet) \rightarrow H^0(\cD^\bullet)\cong D(\A,\m)\xrightarrow{\psi} H^1(\cJ^\bullet) \rightarrow H^1(\cS^\bullet)\rightarrow\cdots,
	\]
	where $\psi$ is defined on $\theta\in D(\A,\m)$ as $\psi(\theta)=\sum_{H\in L_1} \theta(\alpha_H)\in \oplus_{H\in L_1} J(H)$.  The map $\psi$ is an isomorphism if (and only if) $\A$ is essential and formal.
\end{remark}

\begin{remark}\label{rem:DJiso}
If $\A$ is essential and $k$-formal for all $k\ge 2$, then the long exact sequence from Remark~\ref{rem:LES} breaks into isomorphisms $H^i(\cD^\bullet(\A,\m))\cong H^i(\cJ^\bullet(\A,\m))$ for $i\ge 0$ (by Remark~\ref{rem:FormalS}).  In particular, if we wish to determine free multiplicities on an arrangement, we may assume by Corollary~\ref{cor:MultifreeFormal} that $\A$ is $k$-formal for all $k\ge 2$, hence the isomorphism $H^i(\cD^\bullet(\A,\m))\cong H^i(\cJ^\bullet(\A,\m))$ holds for $i\ge 0$.
\end{remark}

\begin{lem}\label{lem:ComputingJk}
Let $(\A,\m)$ be a multi-arrangement.  If $H\in L_1$, then set $J(H)=J_1(\A_H,\m(H))=\langle \alpha_H^{\m(H)}\rangle$.  If $X\in L_k$ where $k>1$, then the module $J_k(\A_X,\m_X)$ satisfies
\[
\begin{array}{rl}
J_k(\A_X,\m_X) &=\delta^{k-1}_S\left(\bigoplus\limits_{\substack{Y\in L_{k-1}\\ X<Y}} J_{k-1}(\A_Y,\m_Y) \right)\\[20 pt]
 & =\sum\limits_{\substack{Y\in L_{k-1}\\ X<Y}} \phi^S_{k-1}(X)(J_{k-1}(\A_Y,\m_Y))
\end{array}
\]
with $\delta^{k-1}_S:\cS^{k-1}\rightarrow\cS^k$ and $\phi_S^k(X):S_k(\A_Y)\rightarrow S_k(\A_X)$ the maps from Definition~\ref{defn:FormalityComplex}.
\end{lem}
\begin{proof}
For simplicity we take $\A_X=\A$, so $\A$ has rank $k$ and $X=\cap_{H\in\A} H$.  The tail end of the short exact sequence of complexes $0\rightarrow \cJ^\bullet\rightarrow \cS^\bullet\rightarrow \cD^\bullet \rightarrow 0$ is shown below.
\begin{center}
\begin{tikzcd}
\cJ^{k-2}\ar{r}{\delta^{k-2}_J}\ar{d} & \cJ^{k-1}=\bigoplus\limits_{Y\in L_{k-1}} J_{k-1}(\A_Y,\m_Y)\ar{r}{\delta^{k-1}_J}\ar{d} & \cJ^k=J_k(\A,\m)\ar{d} \\
\cS^{k-2}\ar{r}{\delta^{k-2}_S}\ar{d} & \cS^{k-1}=\bigoplus\limits_{Y\in L_{k-1}} S_{k-1}(\A_Y)\ar{r}{\delta^{k-1}_S}\ar{d} & \cS^k=S_k(\A)\ar{d}\\
\cD^{k-2}\ar{r}{\delta^{k-2}} & \cD^{k-1}=\bigoplus\limits_{Y\in L_{k-1}} D_{k-1}(\A,\m) \ar{r}{\delta^{k-1}} & \cD^k=D_k(\A,\m)\\
\end{tikzcd}
\end{center}
The differentials $\delta_S^{k-2}$ and $\delta^{k-2}$ factor through $S_{k-1}(\A)$ and $D_{k-1}(\A,\m)$, respectively, by Definition~\ref{defn:DerivationComplex}.  It follows that $H^{k-1}(\cS^\bullet)=H^{k-1}(\cD^\bullet)=H^k(\cS^\bullet)=H^k(\cD^\bullet)=0$ by Definition~\ref{defn:Dk}.  Hence the long exact sequence in cohomology yields that $H^k(\cJ^\bullet)=0$, in other words $\delta^{k-1}_J$ is surjective.  The first equality follows from commutativity of the diagram.  By definition, $\delta^k_J=\tau^J_k=\sum_{Y\in L_{k-1}} \phi_{k-1}^J(X)$.  Since $\phi_{k-1}^J(X)$ is the restriction of $\phi_{k-1}^S(X)$, this proves the second equality.
\end{proof}

From Lemma~\ref{lem:ComputingJk}, we see that in order to explicitly determine the complexes $\cJ^\bullet$ and $\cD^\bullet$, it suffices to determining the maps $\phi^S_{k}(Y)$ for $Y\in L_{k}$, or equivalently to determine the differential $\delta^k_S$ of the complex $\cS^\bullet$.  In \S~\ref{sec:Formal}, we saw that $\cS^\bullet\cong (\cR_\bullet^*)\otimes_{\kk} S$, so the differential $\delta^k_S$ is just the transpose of the differential $\delta_k$ in the complex $\cR_\bullet$.  By examining these matrices as they appear in~\cite{BrandtTerao} and~\cite{StefanFormal}, we obtain the following recipe for constructing $\delta^k_S$.

\begin{lem}\label{lem:SkDifferential}
A matrix for $\delta^k_S$ may be inductively defined as follows.  The matrix for $\delta^0_S$ is the coefficient matrix for $\A$, whose rows give coefficients of the linear forms defining $\A$.  Inductively, $\delta^k_S$ may be represented by a matrix whose rows are naturally grouped according to flats $X\in L_k$.  A row corresponding to $X\in L_k$ encodes a relation among rows of $\delta^{k-1}_S$ which correspond to flats $Y\in L_{k-1}$ so that $Y<X$; the set of all rows corresponding to $X\in L_k$ is a choice of basis for all relations among the rows of $\delta^{k-1}_S$ corresponding to flats $Y\in L_k$ so that $Y<X$.
\end{lem}

\begin{exm}[Points in $\mathbb{P}^1$]\label{ex:PointsP1}
Consider the arrangement $\A$ of $k+2$ points in $\mathbb{P}^1$, corresponding to the product $xy(x-a_1y)\ldots(x-a_ky)$.  Let $H_x=V(x),H_y=V(y),$ and $H_i=V(x-a_iy)$ for $i=1,\ldots,k$.  By Lemma~\ref{lem:SkDifferential}, the complex $\cS^\bullet$ is
\[
0\rightarrow S^2 \xrightarrow{\delta^0_S} S^{k+2} \xrightarrow{\delta^1_S} S^k \rightarrow 0,
\]
where
\[
\delta^0=
\begin{bmatrix}
1 & 0 \\
0 & 1 \\
1 & -a_1\\
\vdots & \vdots\\
1 & -a_k
\end{bmatrix}
\qquad \mbox{and} \qquad
\delta^1=
\begin{bmatrix}
-1 & a_1 & 1 & 0 &\cdots & 0\\
-1 & a_2 & 0 & 1 &\cdots & 0\\
\vdots & \vdots & \vdots & \vdots & & \vdots\\
-1 & a_k & 0 & 0 & \cdots & 1
\end{bmatrix}.
\]
Notice that $S_2(\A)\cong S^k$.  Write $m_x,m_y$ for $\m(H_x),\m(H_y)$, respectively, and $m_i$ for $\m(H_i)$, $i=1,\ldots,k$.  By Lemma~\ref{lem:ComputingJk}, $J_2(\A,\m)=\cJ^2(\A,\m)$ is generated by the columns of the matrix
\[
M=
\begin{bmatrix}
-x^{m_x} & a_1y^{m_y} & (x-a_1y)^{m_1} & 0 &\cdots &0\\
-x^{m_x} & a_2y^{m_y} & 0 & (x-a_2y)^{m_2} & \cdots & 0\\
\vdots & \vdots & \vdots & \vdots & & \vdots\\
-x^{m_x} & a_ky^{m_y} & 0 & 0 & \cdots & (x-a_ky)^{m_k}
\end{bmatrix},
\]
so $D^2(\A,\m)\cong \mbox{coker}(M)$.  Notice that $M$ is a matrix for $\delta^1_J$ with the natural choice of basis for $\bigoplus_{H\in L_1} J(H)\cong \bigoplus_{H\in L_1} S(-\m(H))$.  Hence, by Remark~\ref{rem:DJiso}, we may identify $D(\A,\m)$ with $H^1(\cJ^\bullet,\m)$, which is exactly the syzygies on the columns of $M$ (it is also straightforward to see this from the definition of $D(\A,\m)$).  In particular, if $k=1$ so $\A$ is the $A_2$ braid arrangement, then $D(A_2,\m)$ may be identified with the syzygies on the forms $x^{m_x},y^{m_y},$ and $(x-a_1y)^{m_1}$.  This provides an alternative way to identify the generators and exponents of $(A_2,\m)$, which were originally found in~\cite{Wakamiko} (see~\cite{FatPoints},\cite[Example~3.6,Lemma~4.5]{A3MultiBraid} for more details).
\end{exm}

For an arrangement defined by the vanishing of forms $\alpha_1,\ldots,\alpha_n$, we will write $H_i$ for $V(\alpha_i)$ and denote the flat $H_{i_1}\cap\cdots\cap H_{i_k}$ by the list of indices $i_1\cdots i_k$.  Furthermore, we will denote by $\lt$ the set of rank two flats which are the intersection of at least three hyperplanes.

\begin{exm}[$X_3$ arrangement]\label{ex:X3} Consider the arrangement $\A_t$ defined by the vanishing of the six linear forms
\[
\begin{array}{ll}
\alpha_1=x & \alpha_4=x-t y\\
\alpha_2=y & \alpha_5=x+z\\
\alpha_3=z & \alpha_6=y+z.
\end{array}
\]
The intersection lattice of $\A_t$ is constant as long as $t\neq 0,1$, with six double points and three triple points $\lt=\{124,135,236\}$.  Lemma~\ref{lem:SkDifferential} yields
\[
\cS^\bullet=0\rightarrow S^3\xrightarrow{\delta^0_S} S^6\xrightarrow{\delta^1_S} S^3 \rightarrow 0,
\]
where
\[
\delta^0_S=\bordermatrix{ & x & y & z\cr 
	1 & 1 & 0 & 0 \cr
	2 & 0 & 1 & 0 \cr
	3 & 0 & 0 & 1 \cr
	4 & 1 & -t & 0 \cr
	5 & 1 & 0 & 1 \cr
	6 & 0 & 1 & 1}
\qquad
\delta^1_S=\bordermatrix{ & 1 & 2 & 3 & 4 & 5 & 6 \cr
	124 & 1 & -t & 0 & -1 & 0 & 0 \cr
	135 & 1 & 0 & 1 & 0 & -1 & 0 \cr
	236 & 0 & 1 & 1 & 0 & 0 & -1
}.
\]
This complex is always exact, hence $\A_t$ is always formal for $t\neq 0,1$ by Corollary~\ref{cor:HomologicalCharFormality}.  By Remark~\ref{rem:DJiso}, $H^i(\cD^\bullet)\cong H^i(\cJ^\bullet)$.  By Theorem~\ref{thm:Free}, we may check freeness of $D(\A_t,\m)$ by determining vanishing of $H^1(\cJ^\bullet)$.

Now we consider the complex $\cJ^\bullet$.  Write $J(i)$ for $J_1((\A_t)_{H_i},m_i)=\langle \alpha_i^{m_i}\rangle$.  If $ijk\in\lt$, we write $J(ijk)$ for the ideal $J(i)+J(j)+J(k)$, where $ijk\in\lt$.  Then, by Lemma~\ref{lem:ComputingJk}, $J_2(124,\m)=J(1)-tJ(3)-J(4)=J(1)+J(3)+J(4)=J(134)$.  The same holds for any triple point, so $J_2(ijk,\m)=J(ijk)$ for every $ijk\in\lt$.  So $\cJ^2=\oplus_{ijk\in\lt} J(ijk)$ and
\[
\cJ^\bullet= 0\rightarrow \bigoplus\limits_{i=1}^6 J(i)\xrightarrow{\delta^1_J} \bigoplus\limits_{ijk\in\lt} J(ijk),
\]
where $\delta^1_J$ is the restriction of $\delta^1_S$.  A presentation for $H^2(\cJ^\bullet)$ is worked out in~\cite{X3} and is used to prove that $(\A_t,\m)$ is free if and only if the defining equation has the form $\Q(\A,\m)=x^ny^nz^n(x-ty)(x+z)(y+z)$, where $t^n=1$.  We generalize this result in Theorem~\ref{thm:FreeMultNonFreeTF2}.
\end{exm}

\subsection{Graphic arrangements}  
Let $G$ be a simple graph (no loops or multiple edges) on $\ell$ vertices $\{v_1,\ldots,v_\ell\}$ with edge set $E(G)$, $S=\kk[x_1,\ldots,x_\ell]$ (with $x_i$ corresponding to $v_i$), and set $H_{ij}=V(x_i-x_j)$.  The \textit{graphic arrangement} associated to $G$ is the arrangement $\A_G=\cup_{\{v_i,v_j\}\in E(G)} H_{ij}$; $\A_G$ is a sub-arrangement of the $A_{\ell-1}$.  A multiplicity $\m$ on $\A_G$ is determined by the values $m_{ij}=\m(H_{ij})$ corresponding to edges $\{v_i,v_j\}\in E(G)$.

Recall that the \textit{clique complex} (or \textit{flag complex}) of a graph $G$ is the simplicial complex $\Delta=\Delta(G)$ with an $i$-simplex for every complete graph on $(i-1)$ vertices.

\begin{lem}\label{lem:simpcochain}
The chain complex $\cS^\bullet(\A_G)$ may be identified with the simplicial co-chain complex of $\Delta(G)$ with coefficients in $S$.  Hence $\A_G$ is $k$-formal if and only if $H^i(\Delta(G); S)=0$ for $1\le i\le k-1$.
\end{lem}
\begin{proof}
By~\cite[Lemma~3.1]{StefanFormal}, $\cR_\bullet(\A_G)$ may be identified with the simplicial chain complex of $\Delta(G)$ with coefficients in $\kk$.  Now use the isomorphism $\cS^\bullet\cong(\cR_\bullet)^*\otimes_\kk S$.
\end{proof}

\begin{remark}\label{rem:GraphicTotallyFormal}
Using Lemma~\ref{lem:simpcochain} we may easily see how the notions of $k$-formal for various $k$ are distinct; this was part of the intent of~\cite{StefanFormal}.  This lemma also makes it clear that the condition that $\A_G$ is $k$-formal for $2\le k\le r-1$ is distinct from the condition of being totally formal.  A graphic arrangement $\A_G$ is $k$-formal for $2\le k\le r-1$ if and only if its clique complex $\Delta(G)$ is contractible.  On the other hand, $\A_G$ is totally formal if and only if $G$ is chordal; a much stronger condition which coincides with both freeness and supersolvability of $\A_G$~\cite{StanleySupersolvable}.
\end{remark}

If $\sigma\in\Delta(G)_k$ is a complete graph on the $(k+1)$ vertices $\{v_{i_0},\ldots,v_{i_k}\}$ (where $k\ge 1$), then write $J(\sigma)$ for the ideal generated by the forms $\{(x_{i_s}-x_{i_t})^{m_{i_si_t}} : 0\le s<t\le k\}$.  If $\sigma=\{v_i\}$ is a single vertex, then we take $J(\sigma)=0$. 

%

\begin{prop}\label{prop:GraphicDdescription}
If $G$ is a simple graph, then $\cD^\bullet(\A_G,\m)$ has modules 
\[
\cD^i\cong \bigoplus_{\sigma\in\Delta(G)_i} S/J(\sigma)
\]
and differentials $\delta^i$ induced from the simplicial co-chain complex with coefficients in $S$, which may be identified with $\cS^\bullet(\A_G)$.
\end{prop}
\begin{proof}
Use the identification of the differentials $\delta^i$ in Lemma~\ref{lem:simpcochain} as the simplicial co-chain differential for $\Delta(G)$ and the construction of $J_k((\A_G)_X,\m_X)$ from Lemma~\ref{lem:ComputingJk}.
\end{proof}

\begin{remark}
The chain complex in Proposition~\ref{prop:GraphicDdescription} was introduced in~\cite{GSplinesGraphicArr} by analogy with a natural class of chain complexes in the context of multivariate spline theory~\cite{Homology,LCoho}.  Applying Theorem~\ref{thm:Free} yields the homological characterization of freeness obtained in~\cite[Corollary~5.6]{GSplinesGraphicArr}.
\end{remark}

\begin{remark}
The first non-trivial classification of free multiplicities on a graphic arrangement admitting both free and non-free multiplicities was completed in~\cite{AbeDeletedA3}.  Building on work of Abe, Nuida, and Numata~\cite{AbeSignedEliminable}, the classification of free multiplicities on the $A_3$ braid arrangement has been completed in~\cite{A3MultiBraid}.  The key is a detailed analysis of $H^2(\cD^\bullet(A_3,\m))$, where $\cD^\bullet$ is the complex described in Corollary~\ref{prop:GraphicDdescription}.
\end{remark}

\section{$TF_2$ arrangements}\label{sec:TF2}
In this section we introduce a subset of the totally formal arrangements which we shall call $TF_k$ arrangements.  These are totally formal arrangements which additionally satisfy that $\cS^i(\A)=0$ for $i>k$.  For instance, every totally formal arrangement is $TF_k$ for $k\ge r(\A)$.  A graphic arrangement $\A_G$ is $TF_k$ if and only if $G$ is chordal (see Remark~\ref{rem:GraphicTotallyFormal}) and $\dim(\Delta(G))\le k$.  By Theorem~\ref{thm:Free} and Remark~\ref{rem:DJiso}, freeness of $TF_k$ arrangements is determined by the vanishing of $H^i(\cJ^\bullet)$ for $2\le i\le k$.  In the rest of this section we will assume that $\A$ is a $TF_2$ arrangement of rank at least three.


\subsection{Free $TF_2$ arrangements}
Recall that an arrangement $\A$ is \textit{supersolvable} if there is a filtration $\A_1\subset\cdots\subset\A_r=\A$ satisfying the following rank property (RP) and intersection property (IP):
\begin{itemize}
\item[(RP)] $r(\A_i)=i$ for $i=1,\ldots,r(\A)$.
\item[(IP)] For any $H,H'\in\A_i$ there exists some $H''\in\A_{i-1}$ so that $H\cap H'\subset H''$.
\end{itemize}

\begin{prop}\label{prop:FreeTF2Arrangements}
Let $\A$ be an irreducible $TF_2$ arrangement of rank $r=r(\A)$.  Then
\begin{itemize}
\item $|\A|=r-\#\lt+\sum_{X\in\lt}(|\A_X|-1)$
\item $|\A|\le 1+\sum_{X\in\lt}(|\A_X|-1)$
\item $\#\lt\ge r-1$
\end{itemize}
Furthermore, the following are equivalent.
\begin{enumerate}
\item $\A$ is free
\item $|\A|=1+\sum_{X\in\lt} (|\A_X|-1)$
\item $\#\lt=r-1$
\item $\A$ is supersolvable
\end{enumerate}
In particular, if $\A$ is $TF_2$, its freeness may be determined from $L(\A)$.
\end{prop}
\begin{proof}
The first three bullet points are computed from the Euler characteristic of $\cS^\bullet(\A)$ and $\cJ^\bullet(\A)_1$ as follows.  Since $\A$ is $TF_2$, $\cS^\bullet(\A)$ is a short exact sequence of the form:
\[
0\rightarrow S^\ell=S^r\rightarrow S^{|\A|} \rightarrow \bigoplus_{X\in\lt} S^{|\A_X|-2}\rightarrow 0,
\]
so the alternating sum of the ranks yields $|\A|=r+\sum_{X\in\lt}(|\A_X|-2)=r-\#\lt+\sum_{X\in\lt}(|\A_X|-1)$.  For the second bullet point, $\cJ^\bullet(\A)$ has the form
\[
0\rightarrow \bigoplus\limits_{H\in\A} J(H) \xrightarrow{\delta^1_J} \bigoplus\limits_{X\in\lt} J_2(\A_X) \rightarrow 0.
\]
Since $\ker(\delta^1_J)=D(\A)$ and we assumed $\A$ is irreducible, $\ker(\delta^1_J)_1$ is one dimensional, spanned by the Euler derivation.  We may easily compute $\dim J_2(\A_X)_1=|\A_X|-1$ for $X\in\lt$, hence
\[
\dim H^2(\cJ^\bullet)_1=\sum\limits_{X\in\lt}(|\A_X|-1)-|\A|+1
\]
by computing the Euler characteristic of $\cJ^\bullet_1$.  This must be non-negative, yielding $|\A|\le 1+\sum\limits_{X\in\lt}(|\A_X|-1)$.  The third bullet point follows from putting the first two bullet points together.

Now we prove the equivalent conditions for freeness.  The implication $(4)\implies(1)$ is a well known fact.  Since supersolvability is determined from $L(\A)$, the final statement is immediate from $(4)$.  We first prove $(1)\iff (2)$.  From Theorem~\ref{thm:Free} and Remark~\ref{rem:DJiso}, $\A$ is free if and only if $H^2(\cJ^\bullet)=0$.  From the explicit description in Example~\ref{ex:PointsP1}, we see that $J_2(\A_X)$ is generated in degree one for every $X\in\lt$, as is $J(H)\cong\langle \alpha_H\rangle$ for every $H\in\A$.  So $H^2(\cJ^\bullet)$ must also be generated in degree one since it is a quotient of $\sum_{X\in\lt} J_2(\A_X)$.  From our above computation,
\[
\dim H^2(\cJ^\bullet)_1=\sum\limits_{X\in\lt}(|\A_X|-1)-|\A|+1,
\]
hence $\A$ is free if and only if this expression vanishes, i.e. $|\A|=1+\sum\limits_{X\in\lt}(|\A_X|-1)$.  $(3)$ follows immediately from $(2)$ using the expression $|\A|=r-\#\lt+\sum_{X\in\lt}(|\A_X|-1)$ already proved.  Finally, we show $(3)\implies (4)$.  First, for any $X,X'\in\lt$, we prove there is a sequence $X=X_1,H_1,X_2,\ldots,H_{k-1},X_k=X'$ satisfying
	\begin{enumerate}
		\item $H_i\in\A$ for $i=1,\ldots,k-1$
		\item $X_i\in\lt$ for $i=1,\ldots,k$.
		\item $H_i<X_i$ and $H_{i+1}<X_{i+1}$ in $L(\A)$ for $i=1,\ldots,k-1$.
	\end{enumerate}
	To show this, let $H_1,H_2\in\A_X$ and $H'_1,H'_2\in\A_{X'}$ with corresponding linear forms $\alpha_1,\alpha_2,\alpha'_1,\alpha'_2$.  Complete $\alpha_1,\alpha_2,\alpha'_1$ to a basis $B$ of $V^*$ using defining forms of $\A$ (this is possible because $\A$ is essential).  Adding $\alpha'_2$ to $B$, we see there is a relation of length $r+1$ among the forms $B\cup\{\alpha'_2\}$.  Since $\A$ is formal, this relation can be expressed as a linear combination of relations of length three.  We then read off the sequence $X=X_1,H_1,\ldots,X_k=X'$ from this linear combination of relations of length three.
	
	Now we construct a filtration $\mathcal{F}=\mathcal{F}(\A)=\A_1\subseteq\A_2\subseteq\cdots\subseteq\A_{r}=\A$ of $\A$.  Let $\A_1=H$ for any $H\in\A$, and $\A_2=\A_{X_1}$ for some $X_1\in\lt$ so that $H\in\A_{X_1}$ (by Corollary~\ref{cor:genericHyperplane}, every $H\in\A$ passes through some $X\in\lt$).  Build $\A_{i+1}$ from $\A_i$ for $2\le i\le r$ inductively as follows.  By our above claim, there exists $X_i\in\lt$ so that $\A_i\cap\A_{X_i}\neq\emptyset$.  Then set $\A_{i+1}=\A_i\cup\A_{X_i}$.  This process finishes with $\A_{(r-1)+1}=\A_r$, when we have exhausted $\lt$.  Notice that $\mathcal{F}$ satisfies the intersection property (IP) by construction.  Moreover, $r(\A_i)\le r(\A_{i-1})+1$, hence since the filtration has length $r$ with $\A_r=\A$, we must have $r(\A_i)=i$.  Hence $\mathcal{F}(\A)$ is a supersolvable filtration.
\end{proof}

\subsection{Presentation for $H^2(\cJ^\bullet)$}
Assuming $\A$ is a $TF_2$ arrangement, we now obtain an explicit presentation for $H^2(\cJ^\bullet(\A,\m))$.  Consider the diagram in Figure~\ref{fig:H2Pres},  where the chain complex $\cJ^\bullet$ appears on the right hand side ($\cJ^\bullet$ has only two terms since $\A$ is $TF_2$).  For book-keeping purposes we use the formal symbols $[H]$ and $[X,H]$ (or $[\alpha_H],[X,\alpha_H]$), of degree $\m(H)$, to denote the generators $\alpha_H^{\m(H)}$ of the summands $J(H)=\langle \alpha_H^{\m(H)}\rangle$ which appear in $\bigoplus_{H\in\A} J(H)$ and $\bigoplus_{X\in\lt}\bigoplus_{H<X} J(H)$, respectively.  With this notation, the map $\psi_X:D(\A_X,\m_X)\rightarrow\bigoplus_{H<X}J(X)$ in Figure~\ref{fig:H2Pres} is the map $\psi_X(\theta)=\sum_H \dfrac{\theta(\alpha_H)}{\alpha_H^{\m(H)}}[X,H]$ and $\iota:\bigoplus J(H_i)\rightarrow \bigoplus_{X\in\lt}\bigoplus_{X<H_i} J(H_i)$ is the natural inclusion defined by $\iota([H])=\sum_{X\in\lt} \sum_{X<H} [X,H]$ and extended linearly.  The main thing to check for commutativity is that $(\sum (\delta^1_J)_X)\circ\iota=\delta^1_J$, which follows from the definition.

\begin{figure}
	\centering
	\begin{tikzcd}
		& 	\bigoplus\limits_{H\in\A} J(H) \ar{r}{\cong} \ar{d}{\iota} & \bigoplus\limits_{H\in\A} J(H) \ar{d}{\delta^1_J}\\
		\bigoplus\limits_{X\in\lt} D(\A_X,\m_X) \ar{r}{\sum \psi_X} & \bigoplus\limits_{X\in\lt}\left[ \bigoplus\limits_{H<X} J(H)\right] \ar{r}{\sum (\delta^1_J)_X} & \bigoplus\limits_{X\in\lt} J_2(\A_X,\m_X)
	\end{tikzcd}
	\caption{Diagram for Proposition~\ref{prop:H2Pres}}\label{fig:H2Pres}
\end{figure}

\begin{prop}\label{prop:H2Pres}
Suppose $\A$ is an irreducible $TF_2$ arrangement of rank at least three.  Then
\[
H^2(\cJ^\bullet)\cong \mbox{coker}\left(\bigoplus_{X\in\lt} D(\A_X,\m_X)\xrightarrow{\sum \overline{\psi_X}} \mbox{coker}(\iota)\cong S^\kappa\right),
\]
where $\kappa=(\sum_{X\in\lt} |\A_X|)-|\A|$.  Moreover,
\begin{enumerate}
\item $(\A,\m)$ is free if and only if $\sum \overline{\psi_X}$ is surjective.
\item $\kappa>0$, i.e. $|\A|<\sum_{X\in\lt} |\A_X|$.
\item If $|\A|<\sum_{X\in\lt} (|\A_X|-1)$ or equivalently $r<\#\lt$ then $\A$ is totally non-free.  Furthermore in this case every $\A'\in\mathcal{M}(L(\A))$ is totally non-free.
\end{enumerate}
\end{prop}

\begin{remark}
	The presentation in Proposition~\ref{prop:H2Pres} is similar in spirit to a presentation derived in~\cite[Lemma~3.8]{LCoho} for a homology module which governs freeness of bivariate splines on triangulations.
\end{remark}

\begin{proof}
Since the commutative diagram in Figure~\ref{fig:H2Pres} has exact rows, the isomorphism 
\[
H^2(\cJ^\bullet)\cong \mbox{coker}\left(\bigoplus_{X\in\lt} D(\A_X,\m_X)\xrightarrow{\sum \overline{\psi_X}} \mbox{coker}(\iota)\right)
\]
follows from the tail end of the snake lemma.  The statement (1) now follows from the isomorphism $H^1(\cD^\bullet)\cong H^2(\cJ^\bullet)$ and Theorem~\ref{thm:Free}.

The ideals $J(H)\cong \langle \alpha_H^{\m(H)}\rangle$ are principal, so are isomorphic to the polynomial ring $S$ (up to a graded shift).  The rank of $\bigoplus J(H)$ is $|\A|$ and by the definition of the map $\iota$, we see that the kernel is spanned by the basis elements $[H]$ so that $H$ does not pass through any $X\in\lt$.  However, any such hyperplane is a \textit{generic} hyperplane; by Corollary~\ref{cor:genericHyperplane} the existence of such a hyperplane forces $\A$ to be non-formal.  Hence if $\A$ is $TF_2$, $\iota$ is injective.  Since $\bigoplus\limits_{X\in\lt}\left[ \bigoplus\limits_{H_i<X} J(H_i)\right]$ is a free module of rank $\sum_{X\in\lt} |\A_X|$, we have proved that $\mbox{coker}(\iota)\cong S^\kappa$.  The map $\iota$ is surjective if and only if $\kappa=0$, in which case $H^2(\cJ^\bullet)=0$ regardless of the multiplicity $\m$.  In this case $\A$ is totally free; by~\cite{TeraoTotallyFree} $\A$ is a product of one and two dimensional arrangements, violating the assumption that $\A$ is irreducible.  This proves (2).  

For (3), notice that, in order for $D(\A,\m)$ to be free, the image of $\sum \psi_X$ and the image of $\iota$ must span the entire free module $\bigoplus\limits_{X\in\lt}\left[ \bigoplus\limits_{H<X} J(H)\right]$.  Given (1), the image of $\iota$ does not span this entire free module.  This means that there are some basis elements $[X,H]$ of degree $\m(H)$ (for some hyperplane $H$) that remain in $\mbox{coker}(\iota)$.  In order to kill such basis elements, there must be a basis element $\theta_X\in D(\A_X,\m_X)$ of degree $\m(H)$ which does not vanish on $\alpha_H$.  Notice that for a fixed $X\in\lt$, there cannot be two distinct $H,H'\in\A_X$ so that $\deg(\theta_X)=\m(H)$, $\deg(\psi_X)=\m(H')$, with $\theta_X(\alpha_H)\neq 0$ and $\psi_X(\alpha_{H'})\neq 0$ (see Lemma~\ref{lem:Boolean}).  Hence there are at most $\#\lt$ derivations (one per $X\in\lt$) that can have the right form to cancel remaining basis elements of $\mbox{coker}(\iota)$; it follows that if $|\A|+\#\lt<\sum_{X\in\lt}(|\A_X|)$ then $\A$ is totally non-free, proving the first inequality of (3).  The equivalent formulation for the inequality follows from the equation $|\A|=r-\#\lt+\sum_{X\in\lt}(|\A_X|-1)$ from Proposition~\ref{prop:FreeTF2Arrangements}.  For the final statement of (3), it follows from Lemma~\ref{lem:GenericFormal} that $\A'\in\mathcal{M}(\A)$ is $TF_2$ on a Zariski open subset of $\mathcal{M}(L(\A))$.  Hence on this open set, total non-freeness of $\A'$ follows from the same computation.  Moreover, if $\A'$ is in the complement of this open set, $\A'$ is totally non-free by Corollary~\ref{cor:MultifreeFormal}.
\end{proof}

\begin{cor}\label{cor:TF2Restriction}
Suppose $\A$ is a $TF_2$ arrangement with $r(\A)>\#\lt$, and suppose $\mathcal{B}$ is an arrangement of rank four.  If $L(\mathcal{B})$ has two flats $X,Y\in L(\mathcal{B})$ so that $L(\A)\cong [X,Y]$, then $\mathcal{B}$ is not free.
\end{cor}
\begin{proof}
If $L(\A)$ is isomorphic to an interval in $L(\mathcal{B})$, then $\mathcal{B}$ has either a closed sub-arrangement or a restriction which is in $\mathcal{M}(L(A))$.  In either case, the sub-arrangement or restriction is totally non-free by Proposition~\ref{prop:H2Pres}.  If $\mathcal{B}$ is free, any closed sub-arrangement is also free.  Moreover, the restriction of a free arrangement admits a free multiplicity by Theorem~\ref{thm:Yoshinaga}.  Hence $\mathcal{B}$ cannot be free.
\end{proof}

\begin{exm}[Ziegler's Pair]\label{ex:ZieglerPair}
Consider a central arrangement $\A$ of rank three with nine hyperplanes $\alpha_1,\ldots,\alpha_9$ whose lattice has 18 double points and six triple points, explicitly we assume $\lt=\{145,138,256,289,367,479\}$.  This arrangement can be realized as a line arrangement in $\mathbb{P}\kk^2$ as the lines extending the edges of a hexagon, along with three lines joining opposite vertices (thus the set $\lt$ forms the vertices of the hexagon).  Since there is a non-empty Zariski open space of $\mathcal{M}(L)$ on which $\A$ is $TF_2$ an $\#\lt=6>3=r(\A)$, Proposition~\ref{prop:H2Pres} implies that any $\A\in\mathcal{M}(L)$ is totally non-free.  By Corollary~\ref{cor:TF2Restriction}, no $\A\in\mathcal{M}(L)$ can be the restriction of a free arrangement.

This arrangement appears in~\cite{ZieglerMulti} and~\cite{YuzFormal} as an example of the non-combinatorial behavior of the minimal free resolution of $D(\A)$ and the formality of $\A$, respectively.  More precisely, it is known (due to Yuzvinsky~\cite{YuzFormal}, see also~\cite[Example~13]{SchenckComputationsConjectures}) that $\A$ is formal if and only if the points of $\lt$ do not lie on a conic in $\mathbb{P}^2$.  We may compute that $\cS^\bullet$ has the form $0\rightarrow S^3\xrightarrow{\delta^0_S} S^9 \xrightarrow{\delta^1_S} S^6\rightarrow 0$ if the six points do not lie on a conic and $0\rightarrow S^3 \xrightarrow{\delta^0_S} S^9 \xrightarrow{\delta^1_S} S^5\xrightarrow{\delta^2_S} S\rightarrow 0$ if the six points of $\lt$ do lie on a conic ($\delta^1_S$ drops rank).
\end{exm}

\subsection{A codimension two incidence graph}
The data in the presentation of $H^2(\cJ^\bullet)$ in Proposition~\ref{prop:H2Pres} can be combinatorially encoded using the \textit{codimension two incidence graph} of $\A$, which we denote by $G(\A)$.  The graph $G(\A)=(V,E)$ is a bipartite graph whose vertex set is partitioned as $V=\lt\cup\A$.  There is an edge $[X,H]$ between $X\in\lt$ and $H\in\A$ if and only if $H<X$ in $L(\A)$ (notice that we do not include codimension two flats which are intersections of just two hyperplanes).  Moreover, we define the \textit{reduced} codimension two incidence graph $\overline{G}(\A)$ by removing the vertices $H\in V(G(\A))$ of valence one (i.e. removing vertices corresponding to hyperplanes which only pass through a single flat $X\in\lt$).

Now we describe how $G(\A)$ and $\overline{G}(\A)$ are useful in the context of Proposition~\ref{prop:H2Pres}.  Referring to the diagram in Figure~\ref{fig:H2Pres}, consider the sub-module $N$ of $\bigoplus\limits_{X\in\lt}\bigoplus\limits_{H<X} J(H)$ generated by the image of $\iota$ and the image of $\sum\psi_X$.  Since $D(\A_X,\m_X)$ is a free rank two module for every $X\in\lt$, it is generated by two derivations; call these $\theta_X$ and $\psi_X$.  Then $N$ is generated by the columns of a matrix we denote $M=M(\theta_X,\psi_X\mid X\in\lt)$.  The rows of $M$ are naturally indexed by the formal symbols $[X,H]$ corresponding to basis elements of $\bigoplus_{X\in\lt}\bigoplus_{H<X} J(H)$ - equivalently we may assume the rows are indexed by edges of $G(\A)$.  The columns of $M$ are indexed either by hyperplanes $H'\in\A$ (these represent the image of $\iota$, one for each generator of $\bigoplus_{H\in\A}J(H)$) or pairs $(X',\theta_{X'})$ or $(X',\psi_{X'})$ where $X'\in\lt$ and $\theta_{X'},\psi_{X'}$ are generators of $D(\A_{X'},\m_{X'})$ (each pair represents the inclusion of a generator of $D(\A_{X'},\m_{X'})$).  The entries of $M$ are

\begin{center}
\begin{tabular}{rl}
$M_{[X,H],[H']}$ & $=\left\lbrace
\begin{array}{ll}
1 & H'=H\\
0 & H'\neq H
\end{array}
\right.,$ 
\\
$M_{[X,H],[X',\theta_{X'}]}$ & $=\left\lbrace
\begin{array}{ll}
\overline{\theta}_{X'}(\alpha_H) & X'=X\\
0 & X'\neq X
\end{array}
\right.,$\\
$\mbox{and } 
M_{[X,H],[X',\psi_{X'}]}$ & $=\left\lbrace
\begin{array}{ll}
\overline{\psi}_{X'}(\alpha_H) & X'=X\\
0 & X'\neq X
\end{array}
\right.,$
\end{tabular}
\end{center}
where $\overline{\theta}_{X'}(\alpha_H)=\dfrac{\theta_{X'}(\alpha_H)}{\alpha_H^{\m(H)}}$.  

Moreover we can associate the non-zero entries of $M$ to \textit{oriented} and labeled edges of $G(\A)$; the entry $M_{[X,H],[H]}$ corresponds to the orientation $X\to H$ of $[X,H]$ and the entry $M_{[X,H],[X,\theta_X]}$ corresponds to the orientation $H\to X$ of $[X,H]$, along with the label $\theta_X$ on the edge $[X,H]$.  If a vertex $H\in G(\A)$ has valence one, then the corresponding column of $M$ is a generator of $\bigoplus_{X\in\lt}\bigoplus_{H<X} J(H)$; since we are interested in the cokernel of $M$ we may reduce the matrix $M$ to the matrix $\overline{M}$ whose rows are indexed by pairs $[X,H]$ so that $H$ has valence at least two in $G(\A)$.  Clearly the rows of $\overline{M}$ are in bijection with edges of the reduced incidence graph $\overline{G}(\A)$.  Likewise the non-zero entries of $\overline{M}$ correspond to oriented and labeled edges of $\overline{G}(\A)$.

By Proposition~\ref{prop:H2Pres}, $D(\A,\m)$ is free if and only if the columns of $M(\theta_X,\psi_X\mid X\in\lt)$ generate the free module $\bigoplus_{X\in\lt}\bigoplus_{H<X} J(H)$.  As in the proof of Proposition~\ref{prop:H2Pres}, only one generator for each $D(\A_X,\m_X)$, $X\in\lt$, can map to a generator of $\bigoplus_{X\in\lt}\bigoplus_{H<X} J(H)$.  So we will consider sub-matrices of $\overline{M}$ obtained by choosing only a single generator for each $D(\A_X,\m_X)$.  We write $M'=M'(\theta_X \mid X\in\lt)$ for the sub-matrix of $M$ formed by choosing a single generator $\theta_X$ of each $D(\A_X,\m_X)$, $X\in\lt$.  Notice that the columns of $M'$ are now in bijection with the vertices of $\overline{G}$.  In the two cases we consider, maximal minors of $M'$ will be obtained by deleting at most one column.  Thus the terms of a maximal minor of $M'$ are in bijection with orientations of $\overline{G}$ so that every vertex corresponding to a non-deleted column has exactly one incoming edge.  We will use this observation in the next section.

\subsection{Characterization of free multiplicities on $TF_2$ arrangements}
Using Proposition~\ref{prop:H2Pres} we now characterize free multiplicities on $TF_2$ arrangements.  By Proposition~\ref{prop:H2Pres} and Proposition~\ref{prop:FreeTF2Arrangements} we are restricted to the two cases
\begin{itemize}
\item $|\A|=1+\sum_{X\in\lt} (|\A_X|-1)$ (equivalently $\A$ is a supersolvable $TF_2$ arrangement)
\item $|\A|=\sum_{X\in\lt} (|\A_X|-1)$ 
\end{itemize}

\begin{thm}[Free multiplicities on free $TF_2$ arrangements]\label{thm:FreeMultFreeTF2}
Suppose $\A$ is a free, hence supersolvable $TF_2$ arrangement.  By Proposition~\ref{prop:FreeTF2Arrangements}, $\overline{G}=\overline{G}(\A)$ is a tree.  Then $\m$ is a free multiplicity on $\A$ if and only if there is an orientation of $\overline{G}$ satisfying
\begin{enumerate}
\item Every vertex of $\overline{G}$ has at most one incoming edge.
\item The root vertex (no incoming edges) is some $X\in\lt$.
\item Given a directed edge $H\to X$, $\m(H)$ is an exponent of $D(\A_X,\m_X)$
\end{enumerate}
Equivalently, $\m$ is a free multiplicity if and only if there is an ordering $X_1,\ldots,X_{r-1}$ of $\lt$ and a supersolvable filtration $\A_1\subset\cdots\subset\A_r$ satisfying
\begin{enumerate}
\item $\A_2=\A_{X_1}$ and $\A_i=\A_{i-1}\cup\A_{X_{i-1}}$
\item $\A_{X_i}\cap \A_i=\{H_i\}$ for some $H_i\in\A$ ($H_1,\ldots,H_{r-1}$ not necessarily distinct)
\item $\m(H_i)$ is an exponent of $D(X_i,\m_{X_i})$
\end{enumerate}
\end{thm}
\begin{proof}
By Proposition~\ref{prop:H2Pres} and the preceding discussion, $D(\A,\m)$ is free if and only if there are derivations $\theta_X\in D(\A_X,\m_X)$ so that the columns of $\overline{M'}=\overline{M'}(\theta_X\mid X\in\lt)$ generate $\bigoplus\limits_{X\in\lt}\bigoplus\limits_{H<X} J(H)$; in other words there should be a maximal minor with determinant equal to a non-zero constant.  By Proposition~\ref{prop:FreeTF2Arrangements}, we have $|\A|=1+\sum_{X\in\lt}(|\A_X|-1)$ or $|\A|+\#\lt=1+\sum_{X\in\lt}|\A_X|$.  It follows that the matrix $\overline{M'}$ has one more column than row; so the maximal minors are obtained by deleting a column of $\overline{M'}$.  We may assume that the deleted column corresponds to some $X\in\lt$.  Since $\overline{G}$ is a tree, an orientation of $\overline{G}$ satisfying that each vertex has at most one incoming edge is equivalent to a choice of root for the tree.  This in turn is equivalent to choosing a maximal minor of $\overline{M}$ (leave out the column corresponding to the root).  The maximal minor chosen in this way has determinant
\[
\prod\limits_{H\to X} \overline{\theta}_X(\alpha_H),
\]
where the product is taken over directed edges $H\to X$ in the directed tree $\overline{G}$.  This expression is a non-zero constant if and only if $\overline{\theta}_X(\alpha_H)$ is a non-zero constant (equivalently $\theta_X(\alpha_H)=\alpha_H^{\m(H)}$ up to constant multiple) for every directed edge $H\to X$.  Since $\A_X$ is not boolean for any $X\in\lt$, we see by Lemma~\ref{lem:Boolean} that $(\A_X,\m_X)$ cannot have an exponent smaller than $\m(H)$, so this is in turn equivalent to $(\A_X,\m_X)$ having an exponent of $\m(H)$ for every directed edge $H\to X$.  This proves the first characterization.  

We now show the second characterization in terms of supersolvable filtrations is equivalent to the first.  Given an orientation of $\overline{G}$, we can build the required filtration by setting $X_1$ equal to the root vertex and inductively selecting $X_{i+1}$ to satisfy 1) $X_i$ and $X_{i+1}$ are both adjacent to some $H\in\overline{G}$ and 2) $X_i\to H\to X_{i+1}$ is a directed path with respect to the chosen orientation on $\overline{G}$.  Conversely, given such a supersolvable filtration, we may orient $\overline{G}$ by taking $X_1$ to be the root.
\end{proof}

\begin{exm}\label{ex:TF2Graphic}
Suppose $\A$ is defined by $xyz(x-y)(y-z)$ (this is the graphic arrangement corresponding to a four-cycle with a chord).  Then $\overline{G}$ consists of two vertices corresponding to the triple points $X_1$ and $X_2$ defined by $xy(x-y)$ and $yz(y-z)$, respectively.  Clearly $\A$ is a supersolvable $TF_2$ arrangement.  By Theorem~\ref{thm:FreeMultFreeTF2}, $(\A,\m)$ is free if and only if either $D(X_1,\m_{X_1})$ or $D(X_2,\m_{X_2})$ has an exponent equal to $\m(y)$.  

If $\kk$ has characteristic zero, this happens if and only if $\m(y)\ge \m(x)+\m(x-y)-1$ or $\m(y)\ge \m(z)+\m(y-z)-1$ (by~\cite{Wakamiko}), which recovers Abe's classification in~\cite{AbeDeletedA3}.  In fact Abe's classification has a natural extension to any graphic $TF_2$ arrangement (these correspond to chordal graphs with two-dimensional clique complex).  For instance, suppose $\A$ is defined by $xyzw(x-y)(y-z)(z-w)$.  Then $\overline{G}(\A)$ has three vertices and Theorem~\ref{thm:FreeMultFreeTF2} combined with the classification in~\cite{Wakamiko} yields that $(\A,\m)$ is free if and only if
\begin{itemize}
\item $\m(y)\ge \m(x)+\m(x-y)-1$ and $\m(z)\ge \m(y)+\m(y-z)-1$ or
\item $\m(y)\ge \m(z)+\m(y-z)-1$ and $\m(z)\ge \m(w)+\m(z-w)-1$ or
\item $\m(y)\ge \m(x)+\m(x-y)-1$ and $\m(z)\ge \m(w)+\m(z-w)-1$.
\end{itemize}
Each of the three possibilities corresponds to a choice of root for $\overline{G}$.

By similar arguments it is not difficult to show that a constant multiplicity of value greater than one is never a free multiplicity on a graphic $TF_2$ arrangement of rank at least three over a field of characteristic zero.  In fact, if the constant multiplicity is free on a graphic arrangement over a field of characteristic zero then it is a product of braid arrangements \cite[Theorem~6.6]{GSplinesGraphicArr}.  In contrast, suppose $\kk$ is a field of characteristic $p$.  Then it is straightforward to check (using Saito's criterion), that 
\[
x^{p^k}\frac{\partial}{\partial x}+y^{p^k}\frac{\partial}{\partial y}\qquad\mbox{and}\qquad
x^{p^{k+1}}\frac{\partial}{\partial x}+y^{p^{k+1}}\frac{\partial}{\partial y}
\]
form a basis for the multi-arrangement defined by $x^{p^k}y^{p^k}(x-y)^{p^k}$ (here $k$ is any positive integer).  It follows from Theorem~\ref{thm:FreeMultFreeTF2} that the constant multiplicity of value $p^k$ is always free on a graphic $TF_2$ arrangement over a field of characteristic $p$.  Ziegler~\cite{ZieglerMatroid} has shown that freeness of simple arrangements may also depend on the characteristic of the field.
\end{exm}

\begin{exm}[Example~\ref{ex:multiplicitiessupersolvable}, continued]\label{ex:multiplicitiessupersolvable0}
Consider the arrangement $\A(\alpha,\beta)$ defined by $xyz(x-\alpha z)(x-\beta z)(y-z)$ where $\alpha,\beta\in\kk$.  This is a $TF_2$ arrangement with two rank two flats in $\lt$: the flat $X_1$ defined by $xz(x-\alpha z)(x-\beta z)$ and the flat $X_2$ defined by $yz(y-z)$.  The reduced graph $\overline{G}(\A)$ consists of the three vertices $H,X_1,X_2$ joined by the two edges $[H,X_1]$ and $[H,X_2]$.  By Theorem~\ref{thm:FreeMultFreeTF2} a multi-arrangement $(\A(\alpha,\beta),\m)$ is free if and only if either $D(\A_{X_1},\m_{X_1})$ or $D(\A_{X_2},\m_{X_2})$ has an exponent of $\m(z)$.  Example~\ref{ex:multiplicitiessupersolvable} continues the analysis for this multi-arrangement.
\end{exm}

\begin{remark}
The characterization in Theorem~\ref{thm:FreeMultFreeTF2} reduces the problem of determining free multiplicities on free $TF_2$ arrangements to the problem of determining when rank two multi-arrangements have an exponent which is equal to the multiplicity of one of its points, which is a difficult problem in general~\cite{DerProjLine}.  Somewhat surprisingly, free multiplicities on non-free $TF_2$ arrangements admit a complete description, at least in characteristic zero.
\end{remark}

Suppose $\A$ is a non-free $TF_2$ arrangement which admits a free multiplicity.  As mentioned earlier, $|\A|=\sum_{X\in\lt}(|\A_X|-1)$ or $|\A|+\#\lt=\sum |\A_X|$.  Since $\overline{G}(\A)$ is connected (see the proof of Proposition~\ref{prop:FreeTF2Arrangements}) and $\overline{G}(\A)$ has as many vertices as edges, there is a unique cycle in $\overline{G}(\A)$.  Write $C={H_0,X_0,H_1,X_1,\ldots,H_{k-1},X_{k-1},H_0}$ for this cycle, and let $\alpha_0,\ldots,\alpha_{k-1}$ be the corresponding linear forms to $H_0,\ldots,H_{k-1}$.  We observe that the linear forms $\alpha_0,\ldots,\alpha_{k-1}$ must be linearly independent.  To see this, define $\A'=\A_{X_0}\cup\A_{X_1}\cdots\cup\A_{X_{k-2}}$.  Then $\A'$ has rank $k$, contains all hyperplanes defined by $\alpha_0,\ldots,\alpha_{k-1}$, and every defining form of $\A'$ is expressible using $\alpha_0,\ldots,\alpha_{k-1}$.

\begin{thm}[Free multiplicities on non-free $TF_2$ arrangements]\label{thm:FreeMultNonFreeTF2}
Suppose $\A$ is a non-free $TF_2$ arrangement (over a field of characteristic zero) which admits a free multiplicity.  As above, let $C={H_0,X_0,H_1,X_1,\ldots,H_{k-1},X_{k-1},H_0}$ be the unique cycle in $\overline{G}=\overline{G}(\A)$.  Then $\m$ is a free multiplicity on $\A$ if and only if the following conditions are satisfied
\begin{enumerate}
\item $\m(H)=1$ for every $H\in\A$ which is not a vertex of $C$
\item There is an integer $n>0$ so that $\m(H)=n$ for every $H\in\A$ which is a vertex of $C$
\item There are $B_1,\ldots,B_k\in\kk$ satisfying
\begin{itemize}
\item $B_1\cdots B_k\neq 1$ and
\item for every $H\in\A_{X_i}\setminus\{H_i,H_{i+1}\}$ (indices taken modulo $k$), $\alpha_H$ can be written (up to scalar multiple) as $\alpha_H=\alpha_i+\beta^H_i\alpha_{i+1}$ (indices taken modulo $k$) for some $\beta^H_i\in\kk$ satisfying $(\beta^H_i)^{n-1}=B_i$
\end{itemize}
\end{enumerate}
\end{thm}
\begin{proof}
By Proposition~\ref{prop:H2Pres}, we have $|\A|=\sum_{X\in\lt}(|\A_X|-1)$ or $|\A|+\#\lt=\sum |\A_X|$.  So for any choice of $\theta_X$ for every $X\in\lt$ the matrix $\overline{M'}=\overline{M'}(\theta_X\mid X\in\lt)$ is a square matrix.  We find its determinant.  A term of $\det(\overline{M'})$ corresponds to an orientation of $\overline{G}$ in which every vertex has exactly one incoming edge.  Since $\overline{G}$ has a unique cycle, such an orientation of $\overline{G}$ is determined by an orientation of the cycle (every other edge must be directed `away' from the cycle).  Since there are only two choices of orientation for the cycle $C$ which satisfy that every vertex has exactly one incoming edge, there are only two terms in $\det(\overline{M})$.  In fact, if $C={H_0,X_0,H_1,X_1,\ldots,H_{k-1},X_{k-1},H_0}$,
\begin{equation}\label{eq:det}
\det(\overline{M})=\left(\prod\limits_{i=0}^{k-1} \overline{\theta}_{X_i}(\alpha_i)-\prod\limits_{i=0}^{k-1} \overline{\theta}_{X_i}(\alpha_{i+1})\right)\prod_{(H\to X)\notin C} \overline{\theta}_X(\alpha_H),
\end{equation}
where the index $i+1$ is taken modulo $k$ and the directed edge $H\to X$ is the unique direction `away' from the cycle $C$.  From Proposition~\ref{prop:H2Pres}, $(\A,\m)$ is free if and only if there is a choice of $\theta_{X}$ for every $X\in\lt$ so that the determinant~\eqref{eq:det} is a non-zero constant.  We assume that we have such a choice of $\theta_X, X\in\lt$, and deduce the form for $(\A,\m)$ given in the theorem.  Lemma~\ref{lem:Boolean} guarantees that $\overline{\theta}_{X}(\alpha_H)\neq 0$ for any $X\in\lt$ and $H<X$.  Now, fixing an arbitrary $X_i$ in the cycle $C$, we must have $\theta_{X_i}(\alpha_i)=s_i\alpha_i^{\m(\alpha_i)}$ and $\theta_{X_i}(\alpha_{i+1})=t_i\alpha_{i+1}^{\m(\alpha_{i+1})}$ for some non-zero constants $s_i$ and $t_i$.  Hence $\m(\alpha_i)=\m(\alpha_{i+1})=\deg(\theta_{X_i})$.  Reading around the cycle $C$, we see that $\m(\alpha_0)=\m(\alpha_1)=\cdots=\m(\alpha_{k-1})=n$ for some positive integer $n$, proving (2).  

Next again fix an arbitrary $X_i$ in the cycle $C$ and consider the multi-arrangement $(\A_{X_i},\m_{X_i})$.  Since $X_i$ has rank $2$, we may assume $(\A_{X_i},\m_{X_i})$ is defined by $\Q(\A_{X_i},\m_{X_i})=x^ny^n\prod_{j=1}^k(x-a_jy)^{m_j}$ for some integer $k\ge 1$ (since $X\in\lt$) and some non-zero constants $a_1,\ldots,a_k$ (we are writing $m_j$ for $\m(x-a_jy)$).  Notice that $m_j\le n$ for all $j=1,\ldots,k$ since $\theta_{X_i}$ has degree $n$ (this is easily seen by applying Lemma~\ref{lem:Boolean}).  In particular, $(\A_X,\m_X)$ is \textit{balanced} - i.e. $2n\le |\m_X|=2n + \sum_{i=1}^k m_i$.

Next, a result of Abe~\cite[Theorem~1.6]{AbeFreeness3Arrangements} shows that the exponents of a balanced $2$-multi-arrangement differ by at most $|\A|-2=k$.  Write $d^{X_i}_1\ge d^{X_i}_2$ for the exponents of $(\A_{X_i},\m_{X_i})$, and remember that we are assuming $d^{X_i}_2=\deg(\theta_{X_i})=n$.  From Abe's result we get that $|d^{X_i}_1-d^{X_i}_2|=d^{X_i}_1-n\le k$, so $d^{X_i}_1\le n+k$.  But $|\m_{X_i}|=2n+\sum_{j=1}^k m_j=n+d^{X_i}_2$, so  $d^{X_i}_2=n+\sum_{i=1}^k m_i\le n+k$ (this last inequality follows from the previous sentence).  Since $m_j\ge 1$ for every $j$, we must have $m_j=1$ for each $j=1,\ldots,k$.  Now, applying Lemma~\ref{lem:nn11exp} implies that $a_1^{n-1}=\cdots=a_k^{n-1}$.  This yields the second bullet point under (3).  

As remarked just prior to the statement of Theorem~\ref{thm:FreeMultNonFreeTF2}, $\alpha_0,\cdots,\alpha_{k-1}$ are linearly independent.  Change coordinates so that $\alpha_0=x_0,\ldots,\alpha_{k-1}=x_{k-1}$.  Lemma~\ref{lem:nn11exp} again yields that the derivation $\theta_{X_i}$ has the form $\theta_{X_i}=x_i^n\frac{\partial}{\partial x_i}+B_ix_{i+1}^n\frac{\partial}{\partial x_{i+1}}$.  Plugging this into equation~\eqref{eq:det} yields
\begin{equation}\label{eq:det2}
\det(\overline{M})=\left(1-\prod_{i=0}^{k-1} B_i\right)\prod_{(H\to X)\notin C} \overline{\theta}_X(\alpha_H),
\end{equation}
yielding the first bullet point under (3) since this must be a \textit{non-zero} constant.

Now we prove (1).  If $H\in\A$ is not a vertex of $C$ but there is some $X\in C$ so that $H<X$, then $H\in\A_X$ and $\m(H)=1$ since $H\notin C$.  So suppose $H\in\A$ but $H\nless X$ for any $X\in C$.  Then $H<X$ for some $X\in\lt$, and $X\notin C$.  Then there is a unique $H'$ so that $H'$ is closer to $C$ than $X$ as vertices of $\overline{G}$.  Thus $H'\to X$ is a directed edge in any orientation of $\overline{G}$ satisfying that every vertex has a unique incoming edge.  Thus $\theta_X(\alpha_H)$ appears in the expression of Equation~\eqref{eq:det2} and $\theta_X(\alpha_H)=\alpha_H^{\m(H)}=\alpha_H$ (up to constant multiple, since we assume the right hand side of Equation~\eqref{eq:det} is a non-zero constant).  It follows from Lemma~\ref{lem:Boolean} that $(A_X,\m_X)$ is simple, i.e. $\m_X\equiv 1$.  Hence $\m(H)=1$ as well.

Finally, suppose $\A$ is a non-free $TF_2$ arrangement and $(\A,\m)$ has the form indicated in the statement of the theorem.  Then clearly $\det(\overline{M})$ is a non-zero constant by equation~\eqref{eq:det2}, so $(\A,\m)$ is free by Proposition~\ref{prop:H2Pres}.
\end{proof}

\begin{exm}[Example~\ref{ex:TotallyNonFree}, revisited]\label{ex:TotallyNonfree0}
Consider the arrangement $\A(\alpha,\beta)$ defined by $xyz(x-\alpha y)(x-\beta y)(y-z)(z-x)$, where $\alpha,\beta\in\kk$.  This is a non-free $TF_2$ arrangement with three rank two flats in $\lt$: the flat $X_0$ defined by $xy(x-\alpha y)(x-\beta y)$, the flat $X_1$ defined by $yz(y-z)$, and the flat $X_2$ defined by $xz(x-z)$.  The reduced graph $\overline{G}(\A)$ consists of the cycle $C=\{H_0,X_0,H_1,X_1,H_2,X_2,H_0\}$, where $H_0=V(x),H_1=V(y),$ and $H_2=V(z)$.  By Theorem~\ref{thm:FreeMultFreeTF2} the $(\A(\alpha,\beta),\m)$ is free if and only if $\Q(\A,\m)$ has the form
\[
\Q(\A,\m)=x^ny^nz^n(x-\alpha y)(x-\beta y)(y-z)(z-x),
\]
where $\alpha^{n-1}=\beta^{n-1}\neq 1$.
\end{exm}

\subsection{Further counterexamples to Orlik's conjecture}
In this section we consider the family of arrangements $\A_{r,t}$ with defining polynomial
\[
\Q(\A_{r,t})=x_0\left(\prod\limits_{i=1}^r (x^2_i-x^2_0) \right) (x_1-x_2)\cdots (x_{r-1}-x_r)(x_r-tx_1),
\]
where $t\neq 0\in\kk$.  Write $H_0=V(x_0)$.  The restriction $\A^{H^0}_{r,t}$ has defining polynomial
\[
\Q(\A^{H_0}_{r,t})=\left(\prod\limits_{i=1}^r x_i \right) (x_1-x_2)\cdots (x_{r-1}-x_r)(x_r-tx_1).
\]
Ziegler's multi-restriction has the defining polynomial
\[
\Q(\A^{H_0},\m^{H_0})=\left(\prod\limits_{i=1}^r x_i^2 \right) (x_1-x_2)\cdots (x_{r-1}-x_r)(x_r-tx_1)
\]


\begin{prop}\label{prop:Xr}
If $t\neq 1$ and $\kk$ has characteristic zero, the arrangement $\A_{r,t}$ satisfies
\begin{enumerate}
\item $(\A^{H_0}_{r,t},\m^{H_0})$ is free for $t\neq 0,1$,
\item $\A_{r,t}$ is free if and only if $t=-1$,
\item The minimal free resolution of $D(\A^{H_0}_{r,t})$ is a twisted and truncated Koszul complex, $\reg(D(\A^{H_0}_{r,t}))=3$, and $\mbox{pdim}(D(\A^{H_0}_{r,t}))=r-2$ (the maximum).
\end{enumerate}
\end{prop}
\begin{proof}
Write $X_{r,t}$ for $\A_{r,t}^{H_0}$, $\alpha_i$ for $x_i$ ($i=1,\ldots,r$), $\beta_i$ for $x_i-x_{i+1}$ ($i=1,\ldots,r-1$), and $\beta_r$ for $x_r-tx_1$.  The space of all relations on the linear forms of $X_r$ is an $r$-dimensional space.  Write $Y_i$ for the `triple flat' of codimension two given by the vanishing of the forms $\alpha_i,\alpha_{i+1},\beta_i$ for $i=1,\ldots,r-1$, and write $Y_r$ for the flat determined by $\alpha_1,\alpha_r,\beta_r$.  Clearly $\lt=\{Y_1,\ldots,Y_r\}$ and it is not difficult to see that each $Y_i$ contributes one relation to the relation space and they are all linearly independent, hence $X_{r,t}$ is a $TF_2$ arrangement.  Since $\#\lt=r$, the rank of $X_{r,t}$, it follows from Theorem~\ref{thm:FreeMultNonFreeTF2} that $\m^{H_0}$ is a free multiplicity on $X_{r,t}$, proving (1).

For (2), we use Theorem~\ref{thm:Yoshinaga}.  We already have $(\A^{H_0}_{r,t},\m^{H_0})$ free by (1), so we consider local freeness of $\A_{r,t}$ along $H_0$.  If $t\neq -1$, then the closed sub-arrangement with defining equation
\[
(x_1^2-x_0^2)(x_r^2-x_0^2)(x_r-tx_1)x_0
\]
is not free, so neither is $\A_{r,t}$.  So we need to prove local freeness when $t=-1$.  The closed sub-arrangements of $\A_{r,-1}$ along $H_0$ are isomorphic to $A_1\times A_1\times A_1,A_1\times A_2,A_3$ with a hyperplane removed (the \textit{deleted} $A_3$ arrangement), or $A_3$.  Since these are all free, $\A_{r,-1}$ is free by Theorem~\ref{thm:Yoshinaga}.

For (3), we use the presentation from Proposition~\ref{prop:H2Pres}.  We consider only the case $\m\equiv 1$.  As in Proposition~\ref{prop:H2Pres}, write formal symbols $[H]$ (or $[\alpha_H]$) for the generator of $J(H)=\langle \alpha_H \rangle$ and $[X,H]$ (or $[X,\alpha_H]$) for the generator of $J(H)$ inside the direct sum $\bigoplus_{X\in\lt}\bigoplus_{H<X} J(H)$.  In the case of $X_{r,t}$, the map $\iota:\bigoplus J(H)\rightarrow\bigoplus_{X,H} J(H)$ has the form $\iota([\alpha_i])=[Y_i,\alpha_i]+[Y_{i+1},\alpha_i]$ for $i=1,\ldots,r-1$, $\iota([\alpha_r])=[Y_r,\alpha_r]+[Y_r,\alpha_1]$, and $\iota([\beta_i])=[Y_i,\beta_i]$.  Hence in $\mbox{coker}(\iota)$, we may disregard the generators corresponding to $[Y_i,\beta_i]$ and we can choose generators $[Y_1,\alpha_1],\cdots,[Y_r,\alpha_r]$ with $[Y_2,\alpha_1]=-[Y_1,\alpha_1]$, etc.  With this choice of basis, we determine that the map $\sum\overline{\psi_X}:\oplus D(\A_X,\m_X)\rightarrow \mbox{coker}(\iota)$ is given on $\theta\in D(\A_{Y_1},\m_{Y_1})$ by $\theta\rightarrow \overline{\theta}(\alpha_1)[Y_1,\alpha_1]+\overline{\theta}(\alpha_2)[Y_1,\alpha_2]=\overline{\theta}(\alpha_1)[Y_1,\alpha_1]-\overline{\theta}(\alpha_2)[Y_2,\alpha_2]$, where $\overline{\theta}(\alpha_i)=\theta(\alpha_i)/\alpha_i$ (and similarly for $\theta\in D(\A_{Y_i},\m_{Y_i}),i>1$).  Thus we may represent the map $\sum \overline{\psi}_X$ by the matrix
\[
\bordermatrix{&\theta_1  &\upsilon_1  & \theta_2 & \upsilon_2& \cdots &\theta_r & \upsilon_r \cr
[Y_1,\alpha_1]  &\overline{\theta}_1(x_1) & \overline{\upsilon}_1(x_1)  &  0 & 0 &\cdots & -\overline{\theta}_r(x_1) & -\overline{\upsilon}_r(x_1) \cr 
[Y_2,\alpha_2]  & -\overline{\theta}_1(x_2) & -\overline{\upsilon}_1(x_2) & \overline{\theta}_2(x_2) & \overline{\upsilon}_2(x_2) & \cdots & 0 & 0\cr 
[Y_3,\alpha_3]  & 0 & 0 & -\overline{\theta}_2(x_3) & -\overline{\upsilon}_2(x_3) & \cdots & 0 & 0 \cr
	\vdots 			   & \vdots & \vdots & \vdots & \vdots & \ddots &\vdots & \vdots \cr
[Y_r,\alpha_r] & 0 & 0 & 0 & 0 & \cdots & \overline{\theta}_{r}(x_r) & \overline{\upsilon}_r(x_r)
}
\]
Now, for $i=1,\ldots,r$, $D(\A_{Y_i})$ is generated by the derivations
\[
\begin{array}{l}
\theta_i=x_i\dfrac{\partial}{\partial x_i}+x_{i+1}\dfrac{\partial}{\partial x_{i+1}}\\
\upsilon_i=x_i^2\dfrac{\partial}{\partial x_i}+x_{i+1}^2\dfrac{\partial}{\partial x_{i+1}}
\end{array}
\]
for $i=1,\ldots,r-1$ and $D(Y_r)$ is generated by
\[
\begin{array}{l}
\theta_r=x_r\dfrac{\partial}{\partial x_r}+x_{1}\dfrac{\partial}{\partial x_{1}}\\
\upsilon_r=x_r^2\dfrac{\partial}{\partial x_r}+tx_1^2\dfrac{\partial}{\partial x_{1}}
\end{array}
\]
So the above matrix simplifies to
\[
M=
\bordermatrix{&\theta_1  &\upsilon_1  & \theta_2 & \upsilon_2& \cdots &\theta_r & \upsilon_r \cr
	[Y_1,\alpha_1]  &1 & x_1  &  0 & 0 &\cdots & -1 & -tx_1 \cr 
	[Y_2,\alpha_2]  & -1 & -x_2 & 1 & x_2 & \cdots & 0 & 0\cr 
	[Y_3,\alpha_3]  & 0 & 0 & -1 & -x_3 & \cdots & 0 & 0 \cr
	\vdots 			   & \vdots & \vdots & \vdots & \vdots & \ddots &\vdots & \vdots \cr
	[Y_r,\alpha_r] & 0 & 0 & 0 & 0 & \cdots & 1 & x_r
}
\]
Notice that in $\mbox{coker}(M)$, the Euler derivations $\theta_1,\ldots,\theta_r$ identify all basis elements $[Y_1,\alpha_1],\cdots,[Y_r,\alpha_1]$ to a single basis element.  Hence
\[
\mbox{coker}(M)\cong H^2(\cJ^\bullet)\cong \frac{S(-1)}{\langle x_1-x_2,x_2-x_3,\ldots,x_{r-1}-x_r,x_r-tx_1\rangle},
\]
where the $S(-1)$ encodes the fact that the degrees of $[Y_i,\alpha_i]$ are all one.  Since $t\neq 0,1$, $H^2(\cJ^\bullet)\cong S/\mathfrak{m}$, where $\mathfrak{m}$ is the maximal ideal of $S$.

Now, applying the snake lemma to the diagram in Figure~\ref{fig:H2Pres} and using the fact that $\iota$ is injective (see the proof of Proposition~\ref{prop:H2Pres}), we get the four-term exact sequence
\[
0\rightarrow D(X_{r,t}) \rightarrow \bigoplus\limits_{Y\in\lt} D((X_{r,t})_Y,\m_Y) \xrightarrow{M} S(-1)^\kappa \rightarrow H^2(\cJ(X_{r,t}))\rightarrow 0,
\]
where $S(-1)^{\kappa}=\mbox{coker}(\iota)$.  Above we noticed this prunes down to
\[
0\rightarrow D(X_{r,t}) \rightarrow S(-1)\oplus S(-2)^{r} \xrightarrow{T} S(-1) \rightarrow \dfrac{S}{\mathfrak{m}}\rightarrow 0,
\]
where $T=\begin{bmatrix} 0 & x_1-x_2 & \cdots & x_r-tx_1\end{bmatrix}$.  It follows that
\[
D(X_{r,t})\cong S(-1) \oplus K_2(\mathfrak{m})(-1),
\]
where $K_2(\mathfrak{m})(-1)$ is the module of second syzygies of $\mathfrak{m}$, twisted by $-1$.  It is well-known that $K_2(\mathfrak{m})$ has $\binom{r}{2}$ generators of degree $2$, so $D(X_{r,t})$ is generated by the Euler derivation along with $\dbinom{r}{2}$ generators of degree $3$.  Its minimal free resolution is given by truncating the Koszul complex at $K_2(\mathfrak{m})$, so it is linear of length $r-2$, the maximum possible.  Since the resolution is linear, $\reg(D(X_{r,t}))=3$, where $\reg$ denotes Castelnuovo-Mumford regularity.  This completes the proof of (3).
\end{proof}

\begin{remark}\label{rem:GeneralizedX3}
If $t\neq 1$, then the only non-boolean generic flats of $X_{r,t}$ are the obvious ones of rank two corresponding to the closed circuits of length three.  Hence the bound on $\mbox{pdim}(X_{r,t})$ given by Corollary~\ref{cor:pdimcircuit} is zero, while $\mbox{pdim}(X_{r,t})=r-2$.  If $t=1$ then we can see that $\beta_1,\ldots,\beta_r$ forms a  closed circuit of length $r$, in which case $\mbox{pdim}(D(X_{r,1},\m))\ge r-3$ by Corollary~\ref{cor:pdimcircuit}.  In fact, if we introduce the extra variable $x_0$ and change coordinates by the rule $x_i\to x_i-x_0$, we see that $X_{r,1}$ is the graphic arrangement corresponding to a wheel with $r$ spokes.  From~\cite[Example~7.1]{GSplinesGraphicArr}, $\mbox{pdim}(D(X_{r,1},\m))=r-3$ for any multiplicity $\m$.
\end{remark}

\section{The case of line arrangements}\label{sec:SyzygiesTeraoConj}

It is well-known that $D(\A)$ may be identified with the module of syzygies on the Jacobian ideal $\mbox{Jac}(\A)$ of the defining polynomial of $\A$; hence $\A$ is free if and only if $\mbox{Jac}(\A)$ is codimension two and Cohen-Macaulay.  In this section we show that, for rank three arrangements, $D(\A,\m)$ may be identified with potentially higher syzygies of a less geometric object.  We use this to give another formulation of Terao's conjecture for lines in $\mathbb{P}^2$.

First, suppose $\A$ is a $TF_2$ arrangement and consider the diagram in Figure~\ref{fig:H2Pres}.  Since $\iota$ is injective (see the proof of Proposition~\ref{prop:H2Pres}) and $H^1(\cJ^\bullet)\cong D(\A,\m)$, the full snake lemma applied to this diagram yields the exact sequence
\[
0\rightarrow D(\A,\m) \rightarrow \bigoplus\limits_{X\in\lt} D(\A_X,\m_X) \rightarrow S^{\kappa} \rightarrow H^2(\cJ^\bullet) \rightarrow 0,
\]
where the inclusion $D(\A,\m)\rightarrow \bigoplus D(\A_X,\m_X)$ is the sum of the restriction maps $D(\A,\m)\rightarrow D(\A_X,\m_X)$ (recall that the isomorphism $D(\A,\m)\cong H^1(\cJ)$ is given by the map $\psi(\theta)=\sum_{H\in L} \theta(\alpha_H)$).  By Theorem~\ref{thm:Free}, $D(\A,\m)$ is free if and only if
\[
0\rightarrow D(\A,\m) \rightarrow \bigoplus\limits_{X\in\lt} D(\A_X,\m_X) \xrightarrow{\sum\overline{\psi_X}} S^{\kappa}\rightarrow 0
\]
is a short exact sequence.  Hence if $D(\A,\m)$ is free we may identify it with the syzygies on a (necessarily non-minimal) set of generators for the free module $S^\kappa$.

Now suppose $\A$ is rank three, irreducible and totally formal but not $TF_2$, so $\cS^3(\A)=S_3(\A)\neq 0$.  We can set up (see Figure~\ref{fig:H2PresC}) a very similar diagram to the one in Figure~\ref{fig:H2Pres}.  All maps in the top two rows of Figure~\ref{fig:H2PresC} are the same as in Figure~\ref{fig:H2Pres}; in particular $\kappa=\sum_{X\in\lt}|\A_X|-|\A|$ just as in Proposition~\ref{prop:H2Pres}.  The chain complex $\cJ^\bullet(\A,\m)$ appears as the right-most column.  The map labeled $q$ is the quotient map.  The existence of the bottom right horizontal map $\Delta:\mbox{coker}(\iota)\rightarrow J_3(\A,\m)$ follows from the commutativity of the upper right square; furthermore $\Delta$ is surjective since $\delta^1_J$ and $\sum (\delta^1_J)_X$ are both surjective.  The lower left map $i:\ker(\Delta)\rightarrow S^\kappa$ is the inclusion and the map $\hat{q}$ is lifted from $q$ in the obvious way.

\begin{figure}
\centering
\begin{tikzcd}
& 	\bigoplus\limits_{i=1}^n J(H_i) \ar{r}{\cong} \ar{d}{\iota} & \bigoplus\limits_{i=1}^n J(H_i) \ar{d}{\delta^1_J}\\
\bigoplus\limits_{X\in\lt} D(\A_X,\m_X) \ar{d}{\hat{q}} \ar{r}{\sum \psi_X} & \bigoplus\limits_{X\in\lt}\left[ \bigoplus\limits_{H_i<X} J(H_i)\right]\ar{d}{q} \ar{r}{\sum (\delta^1_J)_X} & \bigoplus\limits_{X\in\lt} J_2(\A_X,\m_X)\ar{d}{\delta^2_J}\\
\ker(\Delta)\ar{r}{i} & \mbox{coker}(\iota)\cong S^\kappa \ar{r}{\Delta} & J_3(\A,\m)
\end{tikzcd}
\caption{Diagram for Proposition~\ref{prop:H2PresC}}\label{fig:H2PresC}
\end{figure}

\begin{prop}\label{prop:H2PresC}
Let $\A$ be an essential, irreducible, formal arrangement of rank $3$ which is not $TF_2$.  Then
\[
H^2(\cJ)\cong \mbox{coker}\left( \bigoplus\limits_{X\in\lt} D(\A_X,\m_X)\xrightarrow{\hat{q}} \ker(\Delta)\right).
\]
and $D(\A,\m)$ is free if and only if $\hat{q}$ is surjective.  Moreover, $D(\A,\m)$ is free if and only if
\[
0\rightarrow D(\A,\m) \rightarrow \bigoplus\limits_{X\in\lt} D(\A_X,\m_X) \xrightarrow{i\circ\hat{q}} S^\kappa 
\]
is exact in the first two positions and $\mbox{coker}(i\circ\hat{q})=J_2(\A,\m)$; i.e. the above sequence is a free resolution for $J_2(\A,\m)$.  Moreover, the left-most inclusion of $D(\A,\m)$ into $\bigoplus D(\A_X,\m_X)$ is given by the sum of natural restriction maps.
\end{prop}
\begin{proof}
The identification of $H^2(\cJ)$ with $\mbox{coker}(\hat{q})$ follows from a long exact sequence in homology.  More precisely, the rows of the diagram in Figure~\ref{fig:H2PresC} are all exact.  Hence we may view this diagram as a short exact sequence of chain complexes; the chain complexes are the columns of the diagram.  As we saw in the proof of Proposition~\ref{prop:H2Pres}, the map $\iota$ is injective so the middle column is exact.  Thus the long exact sequence in homology splits into three isomorphisms.  The first isomorphism yields $H^1(\cJ)\cong \ker(\hat{q})$; which we may read as $D(\A,\m)\cong \ker(\hat{q})$ ($H^1(\cJ)\cong D(\A,\m)$ since $\A$ is essential).  The second isomorphisms yields $H^2(\cJ)\cong \mbox{coker}(\hat{q})$, which is the first statement of the proposition.  The third isomorphism yields $H^3(\cJ)=0$.  Hence by Theorem~\ref{thm:Free}, $D(\A,\m)$ is free if and only if $H^2(\cJ)=0$, if and only if $\mbox{coker}(\hat{q})=0$.

If $\hat{q}$ is surjective (if and only if $D(\A,\m)$ is free), then $\mbox{im}(\hat{q})=\ker(\Delta)$; by our previous identification of $D(\A,\m)$ with $\ker(\hat{q})$ we have freeness of $D(\A,\m)$ if and only if the sequence
\[
0\rightarrow D(\A,\m) \rightarrow \bigoplus\limits_{X\in\lt} D(\A_X,\m) \xrightarrow{i\circ \hat{q}} S^\kappa \xrightarrow{\Delta} J_3(\A,\m) \rightarrow 0
\]
is exact.  Chasing the diagram in Figure~\ref{fig:H2PresC}, and using that the map $D(\A,\m)\rightarrow \bigoplus J(H)$ is given by $\psi(\theta)=\sum \theta(\alpha_H)$, yields that the left-most inclusion is given by the sum of natural restriction maps, so we are done.
\end{proof}

Given a matrix for $\Delta$ in the natural choice of basis, we can identify the columns of $\Delta$ with a (often non-minimal) set of generators for $J_3(\A,\m)$.  Thus $\mbox{ker}(\Delta)$ can be identified with syzygies on this set of generators, which we denote by $\syz(\Delta)$.  In this language, we have the following corollary.

\begin{cor}\label{cor:syzygeticCriterion}
$D(\A,\m)$ is free if and only if $\sum_{X\in\lt} (i\circ\hat{q})(D(A_X,\m_X))$ generates $\syz(\Delta)$.
\end{cor}

\begin{remark}
Proposition~\ref{prop:H2PresC} and Corollary~\ref{cor:syzygeticCriterion} generalize Theorem~3.16 and Corollary~6.3 of~\cite{A3MultiBraid}, where the corresponding statements are worked out for $A_3$ multi-braid arrangements.
\end{remark}

Now consider the case $\m\equiv 1$, which is the setting of Terao's question of whether freeness of $\A$ is combinatorial.  In this case a special role is again played by the Euler derivations in $D(\A_X)$.  In terms of corollary~\ref{cor:syzygeticCriterion}, Euler derivations represent syzygies of degree one, which in turn express redundant generators of $J_3(\A)$ (just like $J_2(\A)$, $J_3(\A)$ is generated in degree one).  Write $\overline{D}(\A)$ for $D(\A)$ modulo the summand generated by the Euler derivation.  Then, for $X\in\lt$, $\overline{D}(\A_X)\cong S(-|\A_X|+1)$, as a graded $S$-module.  Also write $e$ for the rank of the free module spanned by the image of the Euler derivations of $D(\A_X,\m_X)$ inside of $S^\kappa$.  Once we have pruned away the Euler derivations, the chain complex from proposition~\ref{prop:H2PresC} (written as a graded complex of $S$-modules) becomes
\begin{equation}\label{eq:1}
0\rightarrow \overline{D}(\A) \rightarrow \bigoplus\limits_{X\in\lt} S(-|\A_X|+1) \rightarrow S(-1)^{\kappa-e}\rightarrow J_3(\A)\rightarrow 0,
\end{equation}
and the first two maps are now \textit{minimal} (matrices for these maps will have no constants other than $0$).  Since it is shown in Proposition~\ref{prop:FreeTF2Arrangements} that freeness of $TF_2$ arrangements is combinatorial, Terao's question for line arrangements reduces to:

\begin{ques}[Terao's question for line arrangements]\label{ques}
If $\A$ is a line arrangement in $\mathbb{P}^2$ which is not $TF_2$, is exactness of the chain complex~\eqref{eq:1} combinatorial?
\end{ques}

\begin{exm}[$A_3$ braid arrangement]
For $\A=A_3$ braid arrangement defined by the forms $x,y,z,x-y,x-z,y-z$, $J_3(A_3)=\langle x,y,z,x-y,x-z,y-z \rangle$.  The $A_3$ arrangement has four triple points.  The image of the Euler derivations $D(\A_X)$, $X\in\lt$ inside of $S^\kappa=S^{12-6}=S^6$ has rank $3$, corresponding to the three redundant generators of $J_3(\A)$. Pruning off the Euler derivations yields the chain complex
\[
0\rightarrow \overline{D}(\A) \rightarrow S(-2)^4 \rightarrow S(-1)^3 \rightarrow J_3(\A)\rightarrow 0,
\]
which is exact since the Koszul syzygies among $x,y,z$ are obtained from the non-Euler derivations on $D(\A_X)$, $X\in\lt$.  This is not minimal since $D(\A)$ has a generator of degree $2$ which expresses a relation among the four non-Euler derivations around triple points.  Once this generator of degree $2$ is pruned off we obtain the Koszul complex resolving $J_3(\A)$, 
\[
0\rightarrow S(-3) \rightarrow S(-2)^3\rightarrow S(-1)^3 \rightarrow J_3(\A) \rightarrow 0.
\]
As expected, $D(\A)$ is free with exponents $1,2,3$ (the generators of degree $1,2$ were pruned off to produce the minimal resolution).
\end{exm}

\section{Concluding remarks}\label{sec:CR}
We have implemented construction of the chain complexes $\cJ^\bullet,\cS^\bullet,\cD^\bullet$ in Macaulay2.  Instructions for loading the functions and detailed examples may be found at \href{http://math.okstate.edu/people/mdipasq/}{http://math.okstate.edu/people/mdipasq/} under the Research tab.

So far, we have not studied the behavior of the chain complex $\cD^\bullet(\A,\m)$ under deletion and restriction. In particular, we have the following question.

\begin{ques}\label{ques:1}
Is there a short exact sequence of complexes $0\rightarrow\cD^\bullet(\A',\m')\rightarrow\cD^\bullet(\A,\m)\rightarrow \cD^\bullet(\A'',\m^*)\rightarrow 0$ corresponding to a triple $(\A',\A,\A'')$ of arrangements (in the sense of~\cite[Definition~1.14]{Arrangements}), where $\m^*$ is the Euler multiplicity~\cite{EulerMult}?
\end{ques}

The main difficulty here is to construct the maps between these chain complexes.  Constructing such maps would provide a tight relationship to the addition-deletion theorem of~\cite{EulerMult}.  We also are not aware of any relationships between the chain complex $\cD^\bullet(\A,\m)$ and the characteristic polynomial of $(\A,\m)$ or a supersolvable filtration of $\A$.


\appendix

\section{The moduli space of an arrangement}\label{app:Moduli}
In this appendix we briefly summarize the construction of the moduli space of a lattice over an algebraically closed field $\kk$.  Given the intersection lattice $L$ of some central arrangement $\A\subset V\cong\kk^\ell$ with $n$ hyperplanes,we obtain the \textit{moduli space} of $L$ in the following steps:
\begin{enumerate}
\item Fix an ordering $H_1,\ldots,H_n$ of the hyperplanes of $\A$.  Then each flat $X\in L$ can be identified with the tuple of integers $i_1,\ldots,i_j$ where $H_{i_s}<X$ for every $s=1,\ldots,j$.
\item Let $M$ be an $n\times\ell$ coefficient matrix of variables and $\kk[M]$ the polynomial ring in these variables.  The rows of $M$ correspond to the hyperplanes $H_1,\ldots,H_n$, in order.
\item Suppose the flat $X\in L_k$ is defined by hyperplanes $H_{i_1},\ldots,H_{i_j}$, with $j>k$.  Then the $(k+1)\times(k+1)$ minors of the submatrix of $M$ formed by the rows $i_1,\ldots,i_j$ must all vanish.  Let $I\subset\kk[M]$ be the radical of the ideal generated by all of these minors for all flats $X\in L$.
\item Now let $\mathcal{B}$ be the set of all possible tuples of $\ell$ hyperplanes which intersect in only the origin.  Each tuple in $\mathcal{B}$ gives rise to an $\ell\times\ell$ sub-matrix of $M$ whose determinant must not vanish.  Let $J$ be the principal ideal generated by the product of all of these determinants.
\item The quasi-affine variety $\mathcal{V}=\mathcal{V}(L)=V(I)\setminus V(J)\subset \mathcal{M}$, endowed with the Zariski topology, corresponds to coefficient matrices of hyperplane arrangements with intersection lattice $L$.
\item Since the correspondence between a coefficient matrix and a hyperplane arrangement is not one-to-one, the moduli space $\mathcal{M}(L)$ of $L$ is obtained from $\mathcal{V}(L)$ by quotienting out by the action of scaling rows of $M$ and a changing coordinates in $V$.
\end{enumerate}

A property of an arrangement $\A$ is \textit{combinatorial} if it can be determined from its lattice; equivalently if the property holds for all $\A'\in\mathcal{M}(L(\A))$.  One of the key open questions in the theory of arrangements (posed by Terao), is whether freeness of arrangements is combinatorial.  Yuzvinsky~\cite{YuzModuli} has shown that free arrangements with intersection lattice $L$ form a Zariski open subset of $\mathcal{M}(L)$.  It is not difficult to show that a similar condition holds for totally formal arrangements.

\begin{lem}\label{lem:GenericFormal}
If $\A$ is an essential and totally formal arrangement then $\mbox{rank}(\cS^i(\A))$ is determined by $L$ for every $i$.  Moreover, the set of essential totally formal arrangements with intersection lattice $L$ is a Zariski open set in $\mathcal{M}(L)$.
\end{lem}
\begin{proof}
The arrangement $\A$ is essential and totally formal if and only if $\cS^\bullet$ is exact (see Corollary~\ref{cor:HomologicalCharFormality}).  Since $\cS^k(\A)=\bigoplus_{X\in L_k} S^k(\A_X)$, it suffices to show inductively that $\rk(S^k(\A_X))$ is determined from $L(\A_X)$ for $k=\rk(X)$.  If $X\in L(\A)$ has rank one, then $\rk S^1(\A_X)=1$.  Now the result follows inductively on the rank of $\A_X$, using the Euler characteristic of $\cS^\bullet(\A_X)$.  See also Remark~\ref{rem:CombFreeObstFromFormality}.

Now decompose $\V(L)$ into its irreducible components $\V(L)=\V_1\cup\cdots\cup\V_k$;  algebraically, this corresponds to a prime decomposition $I=P_1\cap P_2\cdots \cap P_k$ (recall $I$ is radical) where $\V_i=V(P_i)\setminus V(J)$.  Fix a component $\V_i$ of $\V(L)$ and work in its coordinate ring $R=\kk[M]/P_i$.  In other words, we consider an arrangement $\A$ whose coefficient matrix has entries in the integral domain $R$.  By Lemma~\ref{lem:SkDifferential} the differentials of the chain complex $\cS^\bullet(\A)$ (equivalently the differentials of $\cR_\bullet$) are elements of the rational function field $K=\mbox{frac}(R)$.  By the first statement, we see that the conditions for $\A_X$ to be $k$-formal for every $2\le i\le r(X)-1$ and every $X\in L$ are finitely many maximal rank conditions on the differentials $\delta^i_{S,X}$ for $\cS^\bullet(\A_X)$.  Since maximal rank conditions are given by the non-vanishing of certain minors, this shows that there are finitely many rational functions in $K$ that should not vanish if $\A$ and all its closed sub-arrangements are to be $k$-formal for every $k$.  Lifting this back to $R$ by considering numerators and denominators gives the result for $\V(L)$.  Since the determinants in question are multi-homogeneous in the row variables and quotienting by coordinate changes amounts to determining a scalar value for certain variables, this descends to the moduli space $\mathcal{M}(L)$.
\end{proof}

\begin{remark}
For a rank three arrangement, the condition to be formal is expressed by the non-vanishing of a maximal rank minor of the $\delta^2_S$ differential.  Example~\ref{ex:ZieglerPair} shows that the ranks of the free modules in $\cS^\bullet(\A)$ are not combinatorial, and that the condition to be totally formal can be non-trivial.  For Example~\ref{ex:ZieglerPair}, it can be shown that, aside from the polynomials determining the lattice structure, there is a single irreducible quadratic in the coefficients of the forms of $\A$ whose non-vanishing determines formality.
\end{remark}

\section{Two lemmas for multi-arrangements of points in $\mathbb{P}^1$}\label{app:MultiPoints}
In this appendix we collect two simple lemmas for multi-arrangements of points in $\mathbb{P}^1$.  Lemma~\ref{lem:nn11exp} may also be found in~\cite{DerProjLine}, in slightly less generality.

\begin{lem}\label{lem:Boolean}
Suppose $(\A,\m)$ is a multi-arrangement of $k+2$ points in $\mathbb{P}^1$ defined by forms $\alpha_1,\ldots,\alpha_{k+2}$.  Suppose that, for some $1\le j\le k+2$, $\theta\in D(\A,\m)$ satisfies that $\theta(\alpha_j)=\alpha_j^{\m(\alpha_j)}$ (up to multiplication by a constant).  If $\A$ is not boolean, then $\theta(\alpha_i)\neq 0$ for all $i=1,\ldots,k+2$.
\end{lem}
\begin{proof}
Without loss of generality, suppose $\theta(\alpha_2)=0$ and $\theta(\alpha_1)=\alpha_1^{\m(\alpha_1)}$.  Changing coordinates, we may assume $\alpha_1=x$ and $\alpha_2=y$.  Write $\theta=f\dx+g\dy$ and let $d=\deg(\theta)$.  Since $\theta(x)=x^{d}$ and $\theta(y)=0$, $f=x^{d}$ and $g=0$.  Any other $\alpha_j$ has the form $x+a_jy$ for some non-zero constant $a_j$; thus we have $\theta(\alpha_j)=x^d$.  Since $\theta\in D(\A,\m)$, we must have $\alpha_j\mid x^d$, a contradiction unless $\A$ is boolean.
\end{proof}

\begin{lem}\label{lem:nn11exp}
	Let $n$ be a positive integer and $(\A,\m)$ a muli-arrangement of $k+2\le n+2$ points in $\mathbb{P}^1$ with $Q(\A,\m)=x^{n}y^{n}\prod_{i=1}^k(x-a_iy)$.  Then $(\A,\m)$ has exponents $(n,n+k)$ if and only if $a_1,\ldots,a_k$ are distinct $(n-1)st$ roots of a non-zero constant $\beta$.  In this case, the derivation of degree $n$ has the formula $\theta=x^n\dx+\beta y^n\dy$.
\end{lem}
\begin{proof}
	Suppose $a_1,\ldots,a_k$ are distinct $(n-1)st$ roots of a non-zero constant $\beta$.  Then it is straightforward to check that the derivation $\theta$ indicated in the statement of the lemma is in $D(\A,\m)$.  Clearly there can be no derivation of smaller degree, so $\theta$ is a generator of $D(\A,\m)$.  Since the exponents sum to $|\m|=2n+k$, $(\A,\m)$ has exponents $(n,n+k)$.  Now suppose that $(\A,\m)$ has exponents $(n,n+k)$.  Then there is a derivation $\theta=f\dx+g\dy\in D(\A,\m)$ of degree $n$, the smallest possible degree.  Up to constant multiple, we must have $f=x^n$ and $g=y^n$.  Dividing through by the coefficient on $x^n$, if necessary, we may assume $\theta=x^n\dx+\beta y^n\dy$, where $\beta$ is a constant.  Then $\theta(x-a_iy)=x^n-a_i\beta y^n$ is divisible by $x-a_iy$ if and only if $a_i^n-a_i\beta=0$, or $a_i^{n-1}=\beta$.  Since this holds for $i=1,\ldots,k$, we are done.
\end{proof}

\end{document}